\documentclass[review,sort&compress,3p]{elsarticle}
\usepackage{lineno,hyperref}
\usepackage{graphicx}
\usepackage{amsmath, amssymb, amsfonts, amsthm}
\usepackage{mathrsfs}
\usepackage{url}
\usepackage{latexsym}
\usepackage[ruled,vlined]{algorithm2e}
\usepackage{color}
\usepackage{subfigure}
\usepackage{multirow}
\usepackage{epstopdf}
\usepackage{float}
\usepackage{siunitx}
\usepackage{caption}
\definecolor{darkblue}{rgb}{0,0,0.4}
\newtheorem{thm}{Theorem}[section]

\newtheorem{lem}[thm]{Lemma}
\theoremstyle{remark}
\newtheorem{remark}[thm]{Remark}

\newcommand{\ie}{{\it i.e.}}

\graphicspath{{./figs/}}

\journal{}

\begin{document}
	
	\begin{frontmatter}
		
		\title{A robust and stable phase field method for structural topology optimization}

		\author[a]{Huangxin Chen}
		\ead{chx@xmu.edu.cn}
		
		\author[a]{Piaopiao Dong}
		\ead{dongpiaopiao@stu.xmu.edu.cn}
		
		\address[a]{School of Mathematical Sciences and Fujian Provincial Key Laboratory on Mathematical Modeling and High Performance Scientific Computing, Xiamen University, Fujian, 361005, China}
		
		\author[b,c]{Dong Wang}
		\ead{wangdong@cuhk.edu.cn}
		
		\author[b,c]{Xiao-Ping Wang}
		\ead{wangxiaoping@cuhk.edu.cn}
		
		\address[b]{School of Science and Engineering, The Chinese University of Hong Kong, Shenzhen, Guangdong 518172, China}
		\address[c]{Shenzhen International Center for Industrial and Applied Mathematics, Shenzhen Research Institute of Big Data, Guangdong 518172, China}

		\begin{abstract}
			This paper presents a novel phase-field-based methodology for solving minimum compliance problems in topology optimization under fixed external loads and body forces. The proposed framework characterizes the optimal structure through an order parameter function, analogous to phase-field models in materials science, where the design domain and its boundary are intrinsically represented by the order parameter function. The topology optimization problem is reformulated as a constrained minimization problem with respect to this order parameter, requiring simultaneous satisfaction of three critical properties: bound preservation, volume conservation, and monotonic objective functional decay throughout the optimization process. The principal mathematical challenge arises from handling domain-dependent body forces, which necessitates the development of a constrained optimization framework. To address this, we develop an operator-splitting algorithm incorporating Lagrange multipliers, enhanced by a novel limiter mechanism. This hybrid approach guarantees strict bound preservation, exact volume conservation, and correct objective functional  decaying rate.  Numerical implementation demonstrates the scheme's robustness through comprehensive 2D and 3D benchmarks. 
		\end{abstract}
		
		
	\end{frontmatter}
	
	% \linenumbers

	\section{Introduction}
	
	Topology optimization represents a class of mathematical optimization techniques that determine the optimal material distribution within a prescribed design domain to achieve target performance metrics while satisfying physical constraints. With the rapid advancement of computational capabilities and manufacturing technologies, these methods have found widespread applications across multiple engineering disciplines \cite{bendsoe2003topology}. Among various formulations, the minimum compliance problem in structural topology optimization has attracted particular research attention due to its fundamental importance in mechanical design \cite{XU201338}. Current topology optimization approaches can be broadly categorized into several paradigms, such as the Solid Isotropic Material with Penalization (SIMP) approach \cite{bendsoe2003topology,bruyneel2005note,XU201338}, which employs power-law material interpolation with density filtering,  topological derivatives \cite{novotny2012topological},  level-set approaches \cite{tan2023discontinuous},  the evolutionary structural optimization method \cite{xie1993simple,jiao2012new,jia2020new}, the phase field method \cite{li2022provably,  li2022unconditionally,xie2023effective,xia2023modified,JIN2024112932}, and several others \cite{CAI2020112778,wang2006radial}.

	The phase field method, originally developed for modeling phase transitions in materials science \cite{ALLEN1976425,RN216}, has emerged as a powerful framework for topology optimization. This approach characterizes material distributions through an order parameter $\phi(\mathbf{x})$ that smoothly transitions between solid ($\phi=1$) and void ($\phi=0$) regions, with the interface evolution governed by either the Allen-Cahn or Cahn-Hilliard dynamics. The application of phase field methods to topology optimization was pioneered by \cite{CHEN200157} and \cite{bourdin2006phase}, who first demonstrated their effectiveness for designing maximum stiffness structures under given loads. Subsequent developments have significantly advanced this approach: \cite{TAKEZAWA20102697} introduced a reaction-diffusion formulation incorporating sensitivity-derived double-well potentials, establishing the framework's accuracy for minimum compliance problems. A notable innovation came from \cite{CHOI20112407}, who eliminated the need for double-well potentials by directly using the objective function's derivative as the reaction term, enabling natural hole nucleation in elastic and magnetic field applications. Further refinements were made by \cite{yu2022phase}, who developed unconditionally stable first- and second-order schemes for elastostatic problems through constrained energy modifications.
	
	While these methods successfully address compliance minimization in force-free scenarios, their extension to problems with body forces remains challenging. Although \cite{bruyneel2005note} and \cite{XU201338} have explored topology optimization under body force loads, their approaches fail to guarantee monotonic compliance reduction. This represents a significant limitation, as maintaining such monotonicity while satisfying linear elastic constraints with body forces proves particularly difficult within the phase field framework. 
	
	%The current work aims to bridge this gap by developing a phase field approach that simultaneously ensures:
	%\begin{itemize}
	%    \item Monotonic compliance reduction
	%    \item Exact satisfaction of linear elastic constraints
	%    \item Robust handling of body force effects
	%\end{itemize}
	
	Recent advances in numerical methods for phase field-based topology optimization have yielded significant improvements in solution accuracy and stability. Notably, \cite{yu2023second} developed a second-order energy-stable scheme for Allen-Cahn equations through a novel combination of linear stabilization and Crank-Nicolson discretization. Concurrently, \cite{JIN2024112932} established a provably convergent adaptive phase-field method for structural optimization. However, these approaches still rely on modified objective functionals, leaving the optimization of original objectives as an outstanding challenge. 
	
	Important theoretical breakthroughs have emerged in constrained phase field modeling. The works of \cite{CHENG2022114585} and \cite{cheng2022new} introduced Lagrange multiplier techniques for constructing positivity-preserving and mass-conserving schemes for parabolic equations. This framework was extended by \cite{cheng2023length} to develop length-preserving, energy-dissipative schemes for the Landau-Lifshitz equation, and further generalized by \cite{cheng2025computing} for optimal partition problems with orthogonality-preserving gradient flows.
	
	%Building upon these developments, our work adapts the Lagrange multiplier framework to structural topology optimization, incorporating limiter techniques \cite{liu1996nonoscillatory,zhang2010maximum} to simultaneously enforce:
	%\begin{itemize}
	%    \item Strict bound preservation (0 $\leq$ $\phi$ $\leq$ 1)
	%    \item Exact volume conservation
	%    \item Guaranteed objective functional decay
	%\end{itemize}
	%This unified approach addresses the fundamental challenge of maintaining physical constraints while optimizing the original objective functional.
	
	The primary objective of this work is to develop a novel, provably stable phase field method for structural topology optimization. Our approach combines the Lagrange multiplier framework with Karush-Kuhn-Tucker (KKT) conditions to rigorously enforce three critical constraints: (1) bound preservation, (2) volume conservation, and (3) energy dissipation. To simultaneously satisfy these constraints, we incorporate a limiter mechanism \cite{liu1996nonoscillatory,zhang2010maximum} within a first-order operator splitting scheme, yielding an efficient and accurate numerical algorithm for phase field evolution.
	
	The proposed methodology offers several key advantages over existing approaches:
	\begin{itemize}
		\item \textbf{Constraint enforcement}: The numerical scheme guarantees bound-preserving solutions ($0 \leq \phi \leq 1$), exact volume conservation, and monotonic decrease of the objective functional at each iteration, ensuring physical admissibility throughout the optimization process.
		
		\item \textbf{Physical fidelity}: Our formulation correctly handles the full linear elasticity problem with body forces, overcoming limitations of previous phase field methods that were restricted to special load cases.
		
		\item \textbf{Mathematical consistency}: Unlike conventional approaches that employ modified objective functionals with penalty terms, we directly optimize the original objective function while maintaining strict objective functional decaying properties.
	\end{itemize}
	
	This combination of theoretical guarantees and computational practicality represents a significant advance in phase field-based topology optimization, particularly for problems involving complex loading conditions and strict design constraints. To our knowledge, this is the first work on phase field based approaches for structural topology optimization problems that guarantees that monotonically decay of the original objective functional without any modifications. 
	
	%The main goal of this paper is develop a novel stable phase field method for structural topology optimization. We use the Lagrange multiplier approach and the Karush-Kuhn-Tucker (KKT) conditions to enforce bound, volume, energy dissipative. And we use the limiter to ensure bound-preserving and volume-preserving be satisfied concurrently. We adopt first-order operator splitting scheme to develop efficient and accurate algorithm for the phase field problem.  This numerical algorithm our proposed enjoy the following advantages:
	%	\begin{itemize}
		%	\item The numerical schemes satisfies the nice properties of bound-preserving, volume-preserving, and the non-increasing property of the objective functional at each time step.
		%	\item The structural topology optimization satisfy the linear elastic problem with body force.
		%	\item Instead of the objective functional with penalty term, we use the original objective function. And the numerical schemes satisfy the energy dissipative property.
		%\end{itemize}
		The remainder of this paper is organized as follows. Section \ref{sec:model} presents the mathematical formulation of the phase-field-based topology optimization problem, including the governing equations and constraint formulations. In Section \ref{sec:Num}, we develop our novel numerical framework, detailing the first-order operator splitting scheme with Lagrange multiplier enforcement and analyzing its theoretical properties. Section \ref{sec:results} demonstrates the effectiveness of our approach through comprehensive numerical experiments, including both benchmark problems and practical applications. Finally, Section \ref{sec:include} concludes with a summary and discusses potential extensions for future research.
		
		\section{Model formulation}\label{sec:model}
		\subsection{The original model}
		The minimum compliance problem in topology optimization seeks to find the optimal material distribution within a fixed design domain $\Omega \subset \mathbb{R}^d$ ($d = 2, 3$) that minimizes structural compliance under applied loads. The domain $\Omega$ is subject to: body forces $\mathbf{f}$, surface tractions $\mathbf{s}$ on Neumann boundary $\Gamma_T$, prescribed displacements on Dirichlet boundary $\Gamma_D$, and volume constraint $|\Omega_1| = \beta |\Omega| = V_0$
		where $\Omega_1 \subseteq \Omega$ represents the material phase and $\Omega_2$ denotes the void region. 
		
		Let $\mathbf{u}$ be the displacements, the elasticity problem is
		\begin{equation}\label{equ:ela}
			\left\{
			\begin{aligned}
				-\nabla \cdot (\mathbf{E} \varepsilon(\mathbf{u})) = \mathbf{f},~~&{\rm in}~\Omega,\\
				\mathbf{u} = \mathbf{0},~~&{\rm on} ~ \Gamma_D,\\
				\mathbf{E}\varepsilon(\mathbf{u})\cdot \mathbf{n} = \mathbf{s}, ~~&{\rm on} ~ \Gamma_T,\\
				\mathbf{E}\varepsilon(\mathbf{u})\cdot \mathbf{n} = \mathbf{0}, ~~&{\rm on} ~ \Gamma\setminus(\Gamma_D\cap\Gamma_T).
			\end{aligned}
			\right.
		\end{equation}
		Here, $\varepsilon$ is the strain tensor $\varepsilon(\mathbf{u}) = \frac{1}{2}(\nabla\mathbf{u} + (\nabla\mathbf{u})^T)$, $\sigma = \mathbf{E}: \varepsilon(\mathbf{u})$ represents the stress tensor, and $\mathbf{E}$ is the fourth-order stiffness given by,
		\begin{equation*}
			\mathbf{E}[\mathbf{x}] = \begin{cases}
				\mathbf{E}^0, & \mathbf{x} \in \Omega_1, \quad \text{(solid material)} \\
				\mathbf{E}^{\text{int}}(\mathbf{x}), & \mathbf{x} \in \Gamma_\epsilon, \quad \text{(intermediate phase)} \\
				\mathbf{0}, & \mathbf{x} \in \Omega_2, \quad \text{(void region)}
			\end{cases}
		\end{equation*}
		where $\mathbf{E}^0$  is the constant, positive-definite stiffness tensor of the material, $\Gamma_\epsilon : = \Omega \setminus (\Omega_1 \cup \Omega_2)$ represents the diffuse interface region with thickness $\epsilon$, and $\mathbf{E}^{\text{int}}(\mathbf{x})$ denotes the spatially varying stiffness in the transition zone.
		\begin{remark}
			The introduction of an intermediate phase with smoothly varying stiffness $\mathbf{E}^{\text{int}}(\mathbf{x})$ in the transition zone $\Gamma_\epsilon$ ensures the numerical stability during phase evolution, allows gradual material transition, and naturally emerges from the order parameter function which will be introduced later in the phase-field formulation.
		\end{remark}
		
		For an isotropic linear elastic material, the stress-strain relationship is given by Hooke's law:
		\begin{equation*}%\label{equ:str}
			\sigma^0(\mathbf{u}) = \mathbf{E}^0\varepsilon(\mathbf{u}) = \lambda\,\mathrm{tr}(\varepsilon(\mathbf{u}))\mathbf{I} + 2\mu\varepsilon (\mathbf{u}),
		\end{equation*}
		where $\mathrm{tr}(\varepsilon(\mathbf{u}))$ denotes the trace of the strain tensor, $\mathbf{I}$ is the second-order identity tensor, $\lambda$ and $\mu$ are the $lam\acute{e}$ constants. The $lam\acute{e}$ constants are related to the conventional constants through:
		\begin{equation}\label{equ:lame}
			\begin{cases}
				\lambda = \dfrac{E\nu}{(1+\nu)(1-\nu)}, & \mu = \dfrac{E}{2(1+\nu)} \quad \text{in~2D}\\
				\lambda = \dfrac{E\nu}{(1+\nu)(1-2\nu)}, & \mu = \dfrac{E}{2(1+\nu)} \quad \text{in~3D}
			\end{cases}
		\end{equation}
		where $E > 0$ is Young's modulus, $\nu \in (0, 0.5)$ is Poisson's ratio, and $\mu$ remains consistent between 2D and 3D cases.

		\begin{remark}
			In \cite{smejkal2021unified}, the authors derive the thermodynamic stability of deformable isotropic linear elastic solids, include the lam\'{e} constants in 2D and in 3D. However, in \cite{sadd2009elasticity},  the authors develop reduced two-dimensional problems for the elasticity equations in three-dimensional, namely the plane strain problem and the plane stress problem. From the plane strain problem, the lam\'{e} constant $\lambda$ is given as 
			\begin{equation*}
				\lambda = \frac{E\nu}{(1+\nu)(1-2\nu)}, ~~\text{in~2D},
			\end{equation*}
			and the lam\'{e} constant $\lambda$ is same as the first equation in \eqref{equ:lame} from the plane stress problem. 
			
			In this paper, we treat the two-dimensional problem as an independent formulation rather than a reduced version of the elasticity equations. Therefore, we adopt the 2D lam\'{e} constants from \cite{smejkal2021unified}.
		\end{remark}
				
		The structural compliance is defined as the work done by external forces:
		
		\begin{equation*}%\label{equ:compliance}
			\mathcal{J}(\mathbf{u}) = \underbrace{\int_{\Gamma_T} \mathbf{s} \cdot \mathbf{u}\, ds}_{\text{traction work}} + \underbrace{\int_\Omega \mathbf{f} \cdot \mathbf{u}\, d\mathbf{x}}_{\text{body force work}}
		\end{equation*}
		
		The minimum compliance topology optimization problem seeks to find the optimal material distribution $\Omega_1 \subset \Omega$ that solves:
		\begin{align*}%\label{equ:2}
			&\min_{\Omega_1\in\Omega} \mathcal{J}(\mathbf{u})\nonumber\\
			&\rm subject ~to~(\ref{equ:ela})~ and ~|\Omega_1| = \beta|\Omega| = V_0.
		\end{align*}

		\subsection{The phase field representation}

		In contrast to characteristic-function-based approaches like the prediction-correction iterative convolution-thresholding method \cite{CHEN2024113119}, we employ a phase-field order parameter $\phi(\mathbf{x}) \in [0,1]$ to represent the material distribution:
		
		\begin{equation*}%\label{equ:phase_field}
			\phi(\mathbf{x}) = 
			\begin{cases}
				1 & \mathbf{x} \in \Omega_1 \quad \text{(solid material)} \\
				(0,1) & \mathbf{x} \in \Gamma_\epsilon \quad \text{(diffuse interface)} \\
				0 & \mathbf{x} \in \Omega_2 \quad \text{(void region)}
			\end{cases}
		\end{equation*}
		
		The stiffness tensor and stress field are expressed through a smoothed interpolation:
		
		\begin{align*}
			\mathbf{E}(\phi) &= \left(E_{\min} + (1-E_{\min})\phi^p\right)\mathbf{E}^0 \label{equ:stiff} \\
			\sigma(\mathbf{u},\phi) &= \left(E_{\min} + (1-E_{\min})\phi^p\right)\mathbf{E}^0:\varepsilon(\mathbf{u}) %\label{equ:stress}
		\end{align*}
		where $0 < E_{\min} \ll 1$ prevents numerical singularity and $p$ is chosen to be $3$ as the penalty exponent in SIMP \cite{bendsoe2003topology}. The body force follows a similar interpolation:
		\begin{equation*}%\label{equ:body_force}
			\mathbf{f}(\phi) = \left(f_{\min} + (1-f_{\min})\phi^p\right)\mathbf{f}^0
		\end{equation*}
		with $0 < f_{\min} \ll 1$.
		The optimization is approximately constrained by:
		\begin{equation*}%\label{equ:volume_constraint}
			\int_\Omega \phi\, d\mathbf{x} = \beta|\Omega| = V_0, \quad \beta \in (0,1)
		\end{equation*}
		yielding the admissible design space:
		\begin{equation*}%\label{equ:design_space}
			\Phi = \left\{ \phi \in H^1(\Omega) \mid 0 \leq \phi \leq 1 \text{ a.e.}, \int_\Omega \phi\, d\mathbf{x} = V_0 \right\}.
		\end{equation*}
		
		We formulate the minimum compliance problem using phase-field regularization as the following constrained minimization:
		
		\begin{align}\label{equ:obj}
			&\min_{\phi,\mathbf{u}} J(\phi, \mathbf{u}) = \underbrace{\int_{\Gamma_N} \mathbf{s}\cdot\mathbf{u}\, ds + \int_\Omega \mathbf{f}(\phi)\cdot\mathbf{u}\,d\mathbf{x}}_{\text{Compliance energy}} + \gamma\underbrace{\int_\Omega \left( \frac{\epsilon}{2}|\nabla\phi|^2 + \frac{1}{\epsilon}F(\phi) \right)d\mathbf{x}}_{\text{Phase-field regularization}} \nonumber \\
			&\text{subject to:} \nonumber \\
			&\quad \phi \in \Phi \quad \text{and} \quad \mathbf{u} \in \mathcal{V} \text{ satisfies} 
		\end{align}
		\begin{equation}\label{equ:pha_ela}
			\begin{cases}
				-\nabla \cdot \left(\mathbf{E}(\phi)\varepsilon(\mathbf{u})\right) = \mathbf{f}(\phi) & \text{in } \Omega \\
				\mathbf{u} = \mathbf{0} & \text{on } \Gamma_D \\
				\mathbf{E}(\phi) \varepsilon(\mathbf{u})\cdot\mathbf{n} = \mathbf{s} & \text{on } \Gamma_N \\
				\mathbf{E}(\phi) \varepsilon(\mathbf{u})\cdot\mathbf{n} = \mathbf{0} & \text{on } \partial\Omega\setminus(\Gamma_D\cup\Gamma_N)
			\end{cases}
		\end{equation}
		where $\gamma > 0$ controls the relative weight of interface energy, $F(\phi) = \frac{1}{4}\phi^2(1-\phi)^2$ is the double-well potential, and 
		$$\mathcal{V} = \{\mathbf{v} \in H^1(\Omega)^d \mid \mathbf{v}|_{\Gamma_D} = \mathbf{0}\}$$ is the admissible displacement space. The objective functional comprises compliance terms including mechanical work from tractions and body forces and Ginzburg-Landau free energy that converges to perimeter measure as $\epsilon \to 0^+$ via $\Gamma$-convergence.
		
		We now establish the existence of solutions to the coupled phase-field topology optimization problem defined by (\ref{equ:obj}) and (\ref{equ:pha_ela}). 
		
		\begin{thm}\label{thm1}
			There exists a minimizer $(\phi^*,\textbf{u}^*)$ to the optimization problem (\ref{equ:obj}), \ie
			\begin{equation*}
				\exists(\phi^*,~\textbf{u}^*) \in \Phi \times \mathcal{V},~~ J(\phi^*,\textbf{u}^*) \leq J(\phi,\textbf{u}),\ \  \forall (\phi,\textbf{u})  \in \Phi \times \mathcal{V}. 
			\end{equation*}
		\end{thm}
		
		\begin{proof}
			Define the solution operator $S(\phi) := \mathbf{u}$ where $\mathbf{u}$ solves \eqref{equ:pha_ela}. Let $\{(\phi^k, \mathbf{u}^k)\}_{k\in\mathbb{N}}$ be a minimizing sequence satisfying:
			\begin{equation*}
				\lim_{k\to\infty} J(\phi^k,S(\phi^k)) = \inf_{\phi \in \Phi} J(\phi,S(\phi)).
			\end{equation*}
			
			From the phase-field energy and elastic energy terms, we obtain $\sup_k \|\nabla \phi^k\|_{L^2(\Omega)} < \infty$ from the Ginzburg-Landau energy and $\|\phi^k\|_{L^\infty(\Omega)} \leq 1$ by definition of $\Phi$.  Thus, $\exists \phi^* \in \Phi$ and a subsequence (relabeled) such that:
			\begin{equation*}
				\phi^k \rightharpoonup \phi^* \text{ in } H^1(\Omega), \quad \phi^k \to \phi^* \text{ in } L^2(\Omega).
			\end{equation*}
			
			The weak formulation yields:
			\begin{equation*}
				\int_\Omega \mathbf{E}(\phi^k)\varepsilon (\mathbf{u}^k):\varepsilon(\mathbf{v})\,d\mathbf{x} = \int_{\Gamma_N} \mathbf{s}\cdot\mathbf{v}\,ds + \int_\Omega \mathbf{f}(\phi^k)\cdot\mathbf{v}\,d\mathbf{x}, \quad \forall \mathbf{v} \in \mathcal{V}.
			\end{equation*}
			Using Korn's inequality and the uniform ellipticity of $\mathbf{E}(\phi^k)$, we derive:
			\begin{equation*}
				\|\mathbf{u}^k\|_{H^1(\Omega)} \leq C\left(\|\mathbf{s}\|_{L^2(\Gamma_T)} + \|\mathbf{f}(\phi^k)\|_{L^2(\Omega)}\right) \leq C'.
			\end{equation*}
			Thus, $\exists \mathbf{u}^* \in \mathcal{V}$ and subsequence with:
			\begin{equation*}
				\mathbf{u}^k \rightharpoonup \mathbf{u}^* \text{ in } H^1(\Omega), \quad \mathbf{u}^k \to \mathbf{u}^* \text{ in } L^2(\Omega).
			\end{equation*}
			For any test function $\mathbf{v} \in \mathcal{V}$:
			\begin{align*}
				&\left|\int_\Omega \left(\mathbf{E}(\phi^k)\varepsilon(\mathbf{u}^k) - \mathbf{E}(\phi^*)\varepsilon(\mathbf{u}^*)\right):\varepsilon(\mathbf{v})\,d\mathbf{x}\right| \\
				&\leq \left|\int_\Omega (\mathbf{E}(\phi^k) - \mathbf{E}(\phi^*))\varepsilon(\mathbf{u}^k):\varepsilon(\mathbf{v})\,d\mathbf{x}\right| + \left|\int_\Omega \mathbf{E}(\phi^*)(\varepsilon(\mathbf{u}^k) - \varepsilon(\mathbf{u}^*)):\varepsilon(\mathbf{v})\,d\mathbf{x}\right| \to 0.
			\end{align*}
			Similarly, the body force term converges:
			\begin{equation*}
				\int_\Omega \mathbf{f}(\phi^k)\cdot\mathbf{v}\,d\mathbf{x} \to \int_\Omega \mathbf{f}(\phi^*)\cdot\mathbf{v}\,d\mathbf{x}.
			\end{equation*}
			Thus, $\mathbf{u}^* = S(\phi^*)$.
			Because of the fact that the Ginzburg-Landau energy is weakly lower semicontinuous:
			\begin{equation*}
				\int_\Omega \left(\frac{\epsilon}{2}|\nabla \phi^*|^2 + \frac{1}{\epsilon}F(\phi^*)\right)d\mathbf{x} \leq \liminf_{k\to\infty} \int_\Omega \left(\frac{\epsilon}{2}|\nabla \phi^k|^2 + \frac{1}{\epsilon}F(\phi^k)\right)d\mathbf{x}
			\end{equation*}
			and the compliance terms converge strongly:
			\begin{equation*}
				\int_{\Gamma_N} \mathbf{s}\cdot\mathbf{u}^k\,ds + \int_\Omega \mathbf{f}(\phi^k)\cdot\mathbf{u}^k\,d\mathbf{x} \to \int_{\Gamma_N} \mathbf{s}\cdot\mathbf{u}^*\,ds + \int_\Omega \mathbf{f}(\phi^*)\cdot\mathbf{u}^*\,d\mathbf{x},
			\end{equation*}
			we have that $(\phi^*,\mathbf{u}^*)$ satisfies:
			\begin{equation*}
				J(\phi^*,\mathbf{u}^*) \leq \liminf_{k\to\infty} J(\phi^k,\mathbf{u}^k) = \inf_{(\phi,\mathbf{u}) \in \Phi \times \mathcal{V}} J(\phi,\mathbf{u}),
			\end{equation*}
			establishing the existence of a minimizer.
		\end{proof}

		\subsection{First-order optimality conditions}
		To derive the necessary conditions for optimality, we construct the Lagrangian functional $\Tilde J$ by incorporating all constraints via Lagrange multipliers:
		
		\begin{align}\label{equ:lag}
			\Tilde J(\phi,\mathbf{u},\bar{\mathbf{u}},\lambda,\eta) = & J(\phi,\mathbf{u}) - \underbrace{\int_\Omega \mathbf{E}(\phi)\varepsilon(\mathbf{u}):\varepsilon(\bar{\mathbf{u}})\,d\mathbf{x}}_{\text{Elasticity weak form}} + \underbrace{\int_{\Gamma_T} \mathbf{s}\cdot\bar{\mathbf{u}}\,ds + \int_\Omega \mathbf{f}(\phi)\cdot\bar{\mathbf{u}}\,d\mathbf{x}}_{\text{Loading terms}} \nonumber \\
			& + \underbrace{\lambda\left(\int_\Omega \phi\,d\mathbf{x} - V_0\right)}_{\text{Volume constraint}} + \underbrace{\int_\Omega \eta\phi(1-\phi)\,d\mathbf{x}}_{\text{Bound constraint}}
		\end{align}
		where $\bar{\mathbf{u}} \in \mathcal{V}$ is the adjoint displacement (Lagrange multiplier for the elasticity system),  $\lambda \in \mathbb{R}$ is the multiplier for the volume constraint, and $\eta \in L^2(\Omega)$ is the multiplier for the bound constraint $0 \leq \phi \leq 1$.
		
		In order to get KKT system for (\ref{equ:lag}), we first derive $\textbf{u}$ and $\bar{\textbf{u}}$ satisfying
		\begin{equation}\label{equ:grad}
			\frac{\delta}{\delta\textbf{u}}\Tilde{J}(\phi,\textbf{u},\bar{\textbf{u}},\lambda,\eta) = 0,~~\frac{\delta}{\delta\bar{\textbf{u}}}\Tilde{J}(\phi, \textbf{u},\bar{\textbf{u}},\lambda,\eta) = 0,
		\end{equation}
		for a given $\phi$. The adjoint equation can be deduced as follows:
		\begin{align*}
			\int_{\Omega}\frac{\delta \Tilde{J}}{\delta \textbf{u}}\cdot\textbf{v}d\textbf{x}& = \frac{d}{d\zeta}\Tilde{J}(\textbf{u}+\zeta\textbf{v})\bigg|_{\zeta=0} \\
			& = \int_{\Gamma_T}\textbf{s}\cdot\textbf{v} d\textbf{s} + \int_{\Omega}\textbf{f}(\phi)\cdot \textbf{v}d\textbf{x} -\int_{\Omega}\textbf{E}(\phi)\varepsilon(\textbf{v}):\varepsilon(\bar{\textbf{u}})d\textbf{x}\\
			&=\int_{\Gamma_T}\textbf{s}\cdot\textbf{v} d\textbf{s} + \int_{\Omega}\textbf{f}(\phi)\cdot \textbf{v}d\textbf{x} + \int_\Omega \nabla\cdot(\textbf{E}\varepsilon(\bar{\textbf{u}}))\cdot\textbf{v}d\textbf{x} -\int_{\Gamma_T}(\textbf{E}\varepsilon(\bar{\textbf{u}}))\cdot\textbf{n}\cdot\textbf{v}d\textbf{x},
		\end{align*}
		that is,
		\begin{equation*}%\label{equ:ela1}
			\left\{
			\begin{aligned}
				-\nabla \cdot (\textbf{E}\varepsilon(\bar{\textbf{u}})) = \textbf{f},~~&{\rm in}~\Omega,\\
				\bar{\textbf{u}} = \textbf{0},~~&\rm{on} ~ \Gamma_D,\\
				\textbf{E}\varepsilon(\bar{\textbf{u}})\cdot \textbf{n} = \textbf{s}, ~~&{\rm on} ~ \Gamma_T,\\
				\textbf{E}\varepsilon(\bar{\textbf{u}})\cdot \textbf{n} = \textbf{0}, ~~&{\rm on} ~\Gamma\setminus(\Gamma_D\cap\Gamma_T).
			\end{aligned}
			\right.
		\end{equation*}
		It's easy to see that $\bar{\textbf{u}} = \textbf{u}$, so we simply set $\bar{\textbf{u}} = \textbf{u}$ in the follows.
		
		By the implicit function theorem,  the variation derivative of $\Tilde J$ with respect to $\phi$ can be computed by 
		\begin{align*}
			\int_{\Omega}\frac{\delta\Tilde{J}(\phi, \mathbf{u}(\phi))}{\delta \phi}\psi\, d\textbf{x} &= \frac{d}{d\zeta}\Tilde{J}(\phi+\zeta\psi)\big|_{\zeta=0}\\
			& = 2\int_{\Gamma_T} \textbf{s}\cdot\textbf{u}'(\phi)\psi \, d\textbf{s} + 2\int_{\Omega}\textbf{f}(\phi)\psi\cdot \textbf{u}'(\phi)\psi\, d\textbf{x} + 2\int_{\Omega}\textbf{f}'(\phi)\psi\cdot \textbf{u}\, d\textbf{x}\\
			&~~~~+ \gamma\int_{\Omega}\big( \epsilon\nabla\phi\cdot\nabla\psi + \frac{1}{\epsilon}F'(\phi)\psi \big)\, d\textbf{x} 
			- \int_\Omega \textbf{E}'(\phi)\psi\varepsilon(\textbf{u}):\varepsilon(\textbf{u})\,d\textbf{x} - 2\int_\Omega \textbf{E}(\phi)\varepsilon(\textbf{u}):\varepsilon(\textbf{u}'(\phi)\psi)\,d\textbf{x}\\
			&~~~~+\int_{\Omega}\lambda\psi\, d\textbf{x} +\frac{1}{|\Omega|}\int_{\Omega}\eta(1-2\phi)\psi\, d\textbf{x}.
		\end{align*}
		Here we let $\textbf{u}_\phi := \langle\textbf{u}'(\phi),\, \psi\rangle $, taking the test function $\textbf{u}_\phi$ in \eqref{equ:pha_ela}, we get
		\begin{equation*}%\label{equ:weak_elas}
			\int_{\Omega}\textbf{E}(\phi)\varepsilon(\textbf{u}):\varepsilon(\textbf{u}_\phi)\, d\textbf{x} = \int_{\Gamma_T} \textbf{s}\cdot \textbf{u}_\phi \, d\textbf{s} + \int_{\Omega}\textbf{f}(\phi)\cdot \textbf{u}_\phi\, d\textbf{x}.
		\end{equation*}
		Therefore, we have
		\begin{equation}\label{equ:delta_phi}
			\int_{\Omega}\frac{\delta\Tilde{J}(\phi, \mathbf{u}(\phi))}{\delta \phi}\psi\, d\textbf{x} = \int_{\Omega}\bigg( -\gamma\epsilon\Delta\phi + \frac{\gamma}{\epsilon}F'(\phi) + 2\textbf{f}'(\phi)\cdot \textbf{u} - \textbf{E}'(\phi)\varepsilon(\textbf{u}):\varepsilon(\textbf{u}) + \lambda + \eta (1-2\phi)\bigg)\psi\,d\textbf{x} + \gamma\int_{\partial\Omega}\epsilon\frac{\partial \phi}{\partial \textbf{n}}\cdot\psi\, d\textbf{s}.
		\end{equation}
		
		%
		%The variation of $\Tilde J$ with respect to $\phi$ can be computed by 
		%\begin{align*}
		%\int_{\Omega}\frac{\delta\Tilde{J}}{\delta \phi}\psi d\textbf{x} &= \frac{d}{d\zeta}\Tilde{J}(\phi+\zeta\psi)\big|_{\zeta=0}\\
		%& = \int_{\Omega}\textbf{f}'(\phi)\psi\cdot \textbf{u} d\textbf{x} + \gamma\int_{\Omega}\big( \epsilon\nabla\phi\cdot\nabla\psi + \frac{1}{\epsilon}F'(\phi)\psi \big)d\textbf{x}  - \big(\int_\Omega \textbf{E}'(\phi)\psi\varepsilon(\textbf{u}):\varepsilon(\textbf{u})d\textbf{x} - \int_\Omega\textbf{f}'(\phi)\psi\cdot \textbf{u}d\textbf{x}\big)\\
		%&~~~~+\int_{\Omega}\lambda\psi d\textbf{x} +\int_{\Omega}\eta(1-2\phi)\psi d\textbf{x}\\
		%&=\int_{\Omega}\bigg( -\gamma\epsilon\Delta\phi + \frac{\gamma}{\epsilon}F'(\phi) + 2\textbf{f}'(\phi)\cdot \textbf{u} - \textbf{E}'(\phi)\varepsilon(\textbf{u}):\varepsilon(\textbf{u}) + \lambda + \eta (1-2\phi)\bigg)\psi d\textbf{x} + \gamma\int_{\partial\Omega}\epsilon\frac{\partial \phi}{\partial \textbf{n}}\cdot\psi d\textbf{s}.
		%\end{align*}

		To solve the phase-field optimality condition, we employ a gradient flow approach in artificial time $t$: 
		$\frac{\partial \phi}{\partial t} = -\frac{\delta}{\delta\phi}\Tilde{J}(\phi,\textbf{u},\bar{\textbf{u}},\lambda,\eta)$. The phase field based equation is then given as,
		\begin{equation}\label{equ:phase}
			\left\{\begin{aligned}
				& \frac{\partial \phi}{\partial t} = \gamma\epsilon\Delta\phi-\frac{\gamma}{\epsilon}F'(\phi) + \textbf{E}'(\phi)\varepsilon(\textbf{u}):\varepsilon(\textbf{u})-2\textbf{f}'(\phi)\cdot\textbf{u}-\lambda-\eta(1-2\phi),~~\rm in ~ \Omega,~t>0,\\
				& \phi(\textbf{x},0) = \phi^0(\textbf{x}), ~~\rm in~ \Omega,\\
				& \frac{\partial\phi}{\partial\textbf{n}} = 0, ~~\rm on~\Gamma,\\
				& \lambda\geq 0,~~\int_\Omega\phi d\textbf{x}-V_0 = 0,~~ \lambda\left(\int_\Omega\phi d\textbf{x}-V_0\right) = 0,\\
				& \eta \geq 0,~~\phi(1-\phi)\geq 0,~~\eta\phi(1-\phi)=0.
			\end{aligned}
			\right.
		\end{equation}
		
		\begin{remark}
			It is easy to see that the above problem involves two fundamentally different types of constraints:
			\begin{itemize}
				\item \textit{Local (pointwise) bound constraint}:
				\begin{equation*}%\label{eq:bound_constraint}
					0 \leq \phi(\mathbf{x},t) \leq 1 \quad \forall \mathbf{x} \in \Omega, \; t > 0
				\end{equation*}
				enforced by a space-time dependent Lagrange multiplier $\eta(\mathbf{x},t) \in L^2(\Omega \times \mathbb{R}^+)$ with complementarity conditions:
				\begin{equation*}
					\eta(\mathbf{x},t) \geq 0, \quad \eta(\mathbf{x},t)\phi(\mathbf{x},t)(1-\phi(\mathbf{x},t)) = 0. 
				\end{equation*}
				\item \textit{Global volume constraint}:
				\begin{equation*}%\label{eq:volume_constraint}
					\int_\Omega \phi(\mathbf{x},t)\,d\mathbf{x} = V_0 \quad \forall t > 0
				\end{equation*}
				enforced by a time-dependent scalar Lagrange multiplier $\lambda(t) \in \mathbb{R}$.
			\end{itemize}
		\end{remark}
		
		According to (\ref{equ:pha_ela}), (\ref{equ:grad}) and (\ref{equ:phase}), we show the rate of objective functional decay in the follows.
		
		\begin{thm}\label{thm2}
			For solutions $(\phi,\mathbf{u})$ to the coupled system \eqref{equ:pha_ela} and \eqref{equ:phase}, the compliance functional satisfies:
			\begin{equation*}%\label{equ:diss}
				\frac{d J(\phi,\textbf{u}(\phi))}{dt} =   -\|\phi_t\|^2  \leq 0, ~~t > 0.
			\end{equation*}
		\end{thm}
		\begin{proof}
			From \eqref{equ:delta_phi} and \eqref{equ:phase}, we obtain
			\begin{equation*}
				\frac{d\Tilde{J}(\phi,\textbf{u}(\phi))}{dt} = \bigg(\frac{\delta\Tilde{J}(\phi,\textbf{u}(\phi))}{\delta\phi},\frac{\partial\phi}{\partial t}\bigg) = -\|\phi_t\|^2.
			\end{equation*}
			and
			\begin{align*}
				\frac{d\Tilde{J}(\phi,\textbf{u}(\phi))}{dt} = &\frac{d J(\phi,\textbf{u}(\phi))}{dt} + \frac{d}{dt} \bigg( -\int_{\Omega} \textbf{E}(\phi) \varepsilon(\textbf{u}):\varepsilon(\textbf{u})\, d\textbf{x} + \int_{\Gamma_T}\textbf{s}\cdot \textbf{u}\,ds + \int_{\Omega} \textbf{f}(\phi)\cdot \textbf{u}\, d\textbf{x} \\
				&+ \lambda\big(\int_{\Omega}(\phi)d\textbf{x} - V_0\big) + \int_{\Omega}\eta\phi(1-\phi)\, d\textbf{x} \bigg).
			\end{align*}
			From the constraint \eqref{equ:pha_ela}, we have
			\begin{equation}\label{equ:b}
				\int_\Omega\textbf{E}(\phi)\varepsilon(\textbf{u}):\varepsilon(\textbf{u})\, d\textbf{x}  -\int_{\Gamma_T}\textbf{s}\cdot\textbf{u}\, ds - \int_{\Omega}\textbf{f}(\phi)\cdot\textbf{u}\, d\textbf{x} = 0.
			\end{equation}
			From volume constraint $\int_\Omega \phi \, d\textbf{x}-V_0=0$ and bound constraints $\phi(1-\phi)\geq 0$, we have $\lambda\big( \int_\Omega\phi d\textbf{x} - V_0 \big)=0$ and $\eta\phi(1-\phi) = 0$.
			
			Combining the above results, the proof is completed
			\begin{equation*}\label{equ:d}
				\frac{d}{dt} J(\phi, \textbf{u})=  -\|\phi_t\|^2.
			\end{equation*}
		\end{proof}
		
\begin{remark}
In this section, we derive the first-order optimality conditions and the associated gradient flow system for the objective functional $\Tilde{J}$. From the gradient flow system \eqref{equ:phase} and \eqref{equ:b}, we observe that 
\[
J(\phi, \mathbf{u}) = \Tilde{J}(\phi, \mathbf{u}, \bar{\mathbf{u}}, \lambda, \eta),
\]
confirming that the original objective functional $J(\phi, \mathbf{u})$ is preserved. Notably, the gradient flow system \eqref{equ:phase} incorporates both bound constraints and a global volume constraint. Developing an efficient numerical algorithm to solve this constrained gradient flow system while ensuring the non-increasing property of the objective functional presents a significant challenge.

To the best of our knowledge, existing stable phase-field methods---which guarantee the decay of the objective functional---heavily rely on the self-adjointness of the system and often require modifications to the objective functional.
\end{remark}

		\section{Numerical scheme for the gradient flow}\label{sec:Num}
		In this section, we propose numerical approximations for the gradient flow system (\ref{equ:phase}). We fixed $\Delta t$ as the time step, and $t^n = n \Delta t,~n=0, 1,2,\cdots,N$, where $T$ is the final time and $N=\frac{T}{\Delta t}$. To effectively solve above system, we decouple the computation of displacement filed and the order parameter separately by a first order operator splitting method. 
		
		\subsection{ First-order operator splitting method}\label{sec:splitting}
		Given the phase field distribution $\phi^n$ at time step $n$, we compute $\phi^{n+1}$ through the following sequence of operations:
		
		\textbf{1. Elastic problem solution:}
		The displacement field $\textbf{u}^{n+1}$ is obtained by solving the linear elasticity boundary value problem as discussed above:
		\begin{equation}\label{equ:dis_ela}
			\left\{
			\begin{aligned}
				-\nabla \cdot (\textbf{E}(\phi^n)\varepsilon(\textbf{u})) = \textbf{f}(\phi^n),~~&\rm{in}~\Omega,\\
				\textbf{u} = \textbf{0},~~&\rm{on} ~ \Gamma_D,\\
				\textbf{E}(\phi^n)\varepsilon(\textbf{u})\cdot \textbf{n} = \textbf{s}, ~~&\rm{on} ~ \Gamma_T,\\
				\textbf{E}(\phi^n)\varepsilon(\textbf{u})\cdot \textbf{n} = \textbf{0}, ~~&\rm{on} ~\Gamma\setminus(\Gamma_D\cap\Gamma_T)
			\end{aligned}
			\right.
		\end{equation}
		where the stiffness tensor $\textbf{E}(\phi^n)$ and body force $\textbf{f}(\phi^n)$ are evaluated using the current phase field distribution.
		
		\textbf{2. Phase field evolution:} The intermediate phase field $\Tilde{\phi}^{n+1}$ is computed via a semi-implicit discretization:
		\begin{equation}\label{equ:dis1}
			\frac{\Tilde{\phi}^{n+1}-\phi^n}{\Delta t} - \gamma\epsilon\Delta\Tilde{\phi}^{n+1} = -\frac{\gamma}{\epsilon}F'(\phi^n) + \textbf{E}'(\phi^n)\varepsilon(\textbf{u}^{n+1}):\varepsilon(\textbf{u}^{n+1}) - 2\textbf{f}'(\phi^n)\cdot \textbf{u}^{n+1}.
		\end{equation}
		
		\textbf{3. Bound-preserving projection: } 
		To enforce the bound preserving constraint, we apply a pointwise projection:
		\begin{align*}%\label{equ:dis2}
			&\frac{\mathring{\phi}^{n+1}-\Tilde{\phi}^{n+1}}{\Delta t} = \eta^{n+1}(1-2\mathring{\phi}^{n+1}),\nonumber\\
			&\eta^{n+1}\geq 0,~~\mathring{\phi}^{n+1}(1-\mathring{\phi}^{n+1})\geq 0, ~~\eta^{n+1}\mathring{\phi}^{n+1}(1-\mathring{\phi}^{n+1})=0.
		\end{align*}
		It is equivalent to a simple cut-off approach:
		\begin{equation}\label{equ:dis3}
			(\mathring{\phi}^{n+1},\eta^{n+1}) =
			\begin{cases}
				&(\Tilde{\phi}^{n+1}, ~0),~~0<\Tilde{\phi}^{n+1}< 1,\\
				&(0,~ -\frac{\Tilde{\phi}^{n+1}}{\Delta t}),~~\Tilde{\phi}\leq 0,\\
				&(1,~ \frac{1-\Tilde{\phi}^{n+1}}{-\Delta t }),~~\Tilde{\phi}\geq 1.
			\end{cases}
		\end{equation}
		
		\begin{remark}[Bound-preserving optimization]
			The bound-constrained projection step in \eqref{equ:dis3} can be interpreted variationally as solving the following optimization problem,
			\begin{equation*}
				\left\{
				\begin{aligned}
					&\eta^{n+1} = \min_{\eta\geq 0 } \tilde{J}(\textbf{u},\phi,\lambda,\eta)\\
					&\eta^{n+1}\geq 0,~~\mathring{\phi}^{n+1}(1-\mathring{\phi}^{n+1})\geq 0, ~~\eta^{n+1}\mathring{\phi}^{n+1}(1-\mathring{\phi}^{n+1})=0.
				\end{aligned}\right.
			\end{equation*}
		\end{remark}

		\textbf{4. Volume conservation:} The volume correction is performed by solving $\hat{\phi}^{n+1},~\lambda^{n+1}$ from
		\begin{align*}%\label{equ:dis4}
			&\frac{\hat{\phi}^{n+1}-\mathring{\phi}^{n+1}}{\Delta t} = \lambda^{n+1},\nonumber\\
			&\lambda^{n+1}\geq 0,~~\int_\Omega\hat{\phi}^{n+1}d\textbf{x}-V_0=0,~~\lambda^{n+1}\big(\int_\Omega\hat{\phi}^{n+1}d\textbf{x}-V_0\big)=0,
		\end{align*}
		where $\lambda^{n+1}$ is constant. It is equivalent to
		\begin{align*}%\label{equ:dis5}
			&\hat{\phi}^{n+1} = \mathring{\phi}^{n+1} + \Delta t \lambda^{n+1},\nonumber\\
			&\int_\Omega\hat{\phi}^{n+1}d\textbf{x} = \int_\Omega ( \mathring{\phi}^{n+1} + \Delta t \lambda^{n+1} )d\textbf{x} = V_0,\\
			&\int_{\text{D}_1}\mathring{\phi}^{n+1}d\textbf{x} + \int_{\text{D}_2}(\mathring{\phi}^{n+1} + \Delta t\lambda^{n+1})d\textbf{x}=V_0,\nonumber\\
			&\lambda^{n+1} = \frac{V_0-\int_\Omega\mathring{\phi}^{n+1}d\textbf{x}}{\Delta t |\text{D}_2|},~~~\rm in~D_2.\nonumber
		\end{align*}
		Here, $\text{D}_1$ denotes the domain that the phase field values are equal to $0$ or $1$, $\text{D}_2$ denotes the other domain, \ie
		\begin{equation*}%\label{equ:spaceD}
			\text{D}_1 := \{p\in \mathcal{N},~\mathring{\phi}^{n+1}(p)=0,~\rm or~\mathring{\phi}^{n+1}(p)=1 \},
		\end{equation*}
		$\text{D}_2 := \mathcal{N}~\backslash \text{D}_1$, $\Omega = \text{D}_1\cup \text{D}_2$ and $\text{D}_1\cap \text{D}_2 = \varnothing$, $\mathcal{N}$ is the set of all vertices of the grid.
		Therefore,
		\begin{equation*}%\label{equ:lambda}
			\lambda^{n+1} = \begin{cases}
				&0,~~\rm in ~~D_1,\\
				&\frac{V_0-\int_\Omega\mathring{\phi}^{n+1}d\textbf{x}}{\Delta t |\text{D}_2|},~~~\rm in~D_2,
			\end{cases}
		\end{equation*}
		and we obtain
		\begin{equation*}%\label{equ:phi}
			\hat{\phi}^{n+1} = \mathring{\phi}^{n+1} + \Psi(\textbf{x})\frac{V_0-\int_\Omega\mathring{\phi}^{n+1}d\textbf{x}}{ |\text{D}_2|}.
		\end{equation*}
		where $\Psi(\textbf{x})$ is the indicator function of domain  $\text{D}_2$.

		\textbf{5. Limiter application:} From the above analysis, we find that $\hat{\phi}^{n+1}$ does not satisfy the bound constraint. To simultaneously satisfy both bound and volume constraints, we apply a linear scaling limiter \cite{liu1996nonoscillatory,zhang2010maximum} in $\text{D}_2$.
		Let $\bar{\phi}^{n+1}$ denote the integral average of $\hat{\phi}^{n+1}$ in domain $\text{D}_2$, \ie
		\begin{equation}\label{equ:average}
			\bar{\phi}^{n+1} = \frac{\int_{\text{D}_2}\hat{\phi}^{n+1}d\textbf{x}}{|\text{D}_2|}
		\end{equation}
		and $\phi_{max} = \max_{p\in \text{D}_2}\hat{\phi}^{n+1}(p)$, $\phi_{min} = \min_{p\in \text{D}_2}\hat{\phi}^{n+1}(p)$.
		The limiter can be applied as follows:
		\begin{equation}\label{equ:phi_new}
			\Breve{\phi} = \Psi(\textbf{x})(\theta(\hat{\phi}^{n+1} - \bar{\phi}^{n+1}) + \bar{\phi}^{n+1})+(1-\Psi(\textbf{x}))\hat{\phi}^{n+1},\end{equation}
		where 
		\[ \theta = \min\bigg\{\bigg|\frac{1-\bar{\phi}^{n+1}}{\phi_{max} -\bar{\phi}^{n+1} }\bigg|,~~ \bigg|\frac{-\bar{\phi}^{n+1}}{\phi_{min} -\bar{\phi}^{n+1} }\bigg|,~~ 1 \bigg\}.\]
		
		\begin{lem}\label{rem:1}
			$\Breve{\phi}$ satisfy the boundedness and volume constraint.
		\end{lem}
		\begin{proof}
			From (\ref{equ:average}) and (\ref{equ:phi_new}), we have
			\begin{align*}
				\int_\Omega\Breve{\phi}^{n+1} d\textbf{x}  &=\int_{\text{D}_1}\hat{\phi}^{n+1} d\textbf{x} + \int_{\text{D}_2}\big(\theta(\hat{\phi}^{n+1} - \bar{\phi}^{n+1}) + \bar{\phi}^{n+1} \big)d\textbf{x}\\
				& = \int_{\text{D}_1}\hat{\phi}^{n+1} d\textbf{x} + \theta \int_{\text{D}_2} \hat{\phi}^{n+1}  d\textbf{x}+(1- \theta) \int_{\text{D}_2}\bar{\phi}^{n+1}  d\textbf{x} \\
				& = \int_{\text{D}_1}\hat{\phi}^{n+1} d\textbf{x} + \theta \int_{\text{D}_2} \hat{\phi}^{n+1}  d\textbf{x}+(1- \theta) \int_{\text{D}_2}\hat{\phi}^{n+1}  d\textbf{x} = V_0.
		\end{align*}\end{proof}
		
		The algorithm for problem (\ref{equ:obj}) and (\ref{equ:pha_ela}) is summarized in Algorithm \ref{a:1}.

		\begin{algorithm}[ht!]
			\DontPrintSemicolon
			\KwIn{$\phi^0$: Initial guess, $\epsilon>0$, $\gamma >0$, $\nu$, $E_{min}$, $\beta$, $N_{max}$: the maximum number of iteration.}
			\KwOut{$\Breve{\phi} \in  \Phi$.}
			Initialize $n=0$.\\
			\While{$n < N_{max}$~~$\&$ ~~$ |J(\phi^{n+1},\textbf{u}^{n+1}) - J(\phi^n,\textbf{u}^n)| > tol$}{
				{\bf 1.} For the fixed $\phi^n$, solve \eqref{equ:dis_ela}
				%\begin{align*}\left\{\begin{aligned}
				%		 -\nabla \cdot (\textbf{E}(\phi^n)\varepsilon(\textbf{u})) = \textbf{f}(\phi^n),~~&\rm{in}~\Omega,\\
				%     \textbf{u} = \textbf{0},~~&\rm{on} ~ \Gamma_D,\\
				%     \textbf{E}(\phi^n)\varepsilon(\textbf{u})\cdot \textbf{n} = \textbf{s}, ~~&\rm{on} ~ \Gamma_T,\\
				%     \textbf{E}(\phi^n)\varepsilon(\textbf{u})\cdot \textbf{n} = \textbf{0}, ~~&\rm{on} ~ \Gamma\setminus(\Gamma_D\cap\Gamma_T).
				%    \end{aligned}
			%    \right.\end{align*}\\
		to have $\textbf{u}^{n+1}$.\\
		{\bf 2.} Use $\textbf{u}^{n+1}$ to solve (\ref{equ:dis1}) to obtain $\Tilde{\phi}^{n+1}$.\\
		{\bf 3. Bounded-preserving step.}
		\[(\mathring{\phi}^{n+1},\eta^{n+1}) =
		\begin{cases}
			&(\Tilde{\phi}^{n+1}, ~0),~~0<\Tilde{\phi}^{n+1}< 1,\\
			&(0,~ -\frac{\Tilde{\phi}^{n+1}}{\Delta t}),~~\Tilde{\phi}^{n+1}\leq 0,\\
			&(1,~ \frac{1-\Tilde{\phi}^{n+1}}{-\Delta t }),~~\Tilde{\phi}^{n+1}\geq 1.
		\end{cases}\]
		{\bf 4. Volume-preserving step.}
		\[
		\hat{\phi}^{n+1} = \mathring{\phi}^{n+1} + \Psi(\textbf{x})\frac{V_0-\int_\Omega\mathring{\phi}^{n+1}d\textbf{x}}{ |\text{D}_2|}
		\]
		%where
		%\[\Psi(\textbf{x})=\begin{cases}
			%        &0, ~~\rm in~D_1,\\
			%        &1, ~~\rm in~D_2.
			%    \end{cases}\]
		and
		\begin{equation*}
			\Breve{\phi} = \Psi(\textbf{x})(\theta(\hat{\phi}^{n+1} - \bar{\phi}^{n+1}) + \bar{\phi}^{n+1})+(1-\Psi(\textbf{x}))\hat{\phi}^{n+1},\end{equation*}
		where 
		\[ \theta = \min\bigg\{\bigg|\frac{1-\bar{\phi}^{n+1}}{\phi_{max} -\bar{\phi}^{n+1} }\bigg|,~~ \bigg|\frac{-\bar{\phi}^{n+1}}{\phi_{min} -\bar{\phi}^{n+1} }\bigg|,~~ 1 \bigg\}.\]
		
		Set $n = n+1$.
	}
	\caption{ Bound-preserving step and volume-preserving step.}
	\label{a:1}
\end{algorithm}

\subsection{Objective functional decaying scheme}
While the operator-splitting method described in Section~\ref{sec:splitting} effectively handles bound and volume constraints, it does not guarantee monotonic decay of the objective functional. To enforce this crucial property, we introduce an additional correction step.

The main idea of the numerical algorithm is to regard the property of dissipation rate of  the objective functional value as a nonlinear global constraint. By introducing a spatially independent Lagrange multiplier $\sigma(t)$, we correct $\Breve{\phi}^{n+1}$ to
\begin{equation}\label{equ:e1}
	\phi^{n+1} = \frac{\Breve{\phi}^{n+1} + \sigma(t)}{\int_\Omega\big( \Breve{\phi}^{n+1} + \sigma(t) \big)d\textbf{x}} V_0.
\end{equation}
such that 
\begin{equation*}%\label{equ:e3}
	\frac{J(\phi^{n+1},\textbf{u}^{n+1}) - J(\phi^n,\textbf{u}^n)}{\Delta t} =- \frac{1}{\Delta t^2}\|\phi^{n+1} - \phi^n\|^2.
\end{equation*}
$\phi^{n+1}$ can be corrected by the root of  the following equation:
\begin{equation}\label{equ:nonlinear}
	F(\sigma):= J(\phi^{n+1},\textbf{u}^{n+1}) - J(\phi^n,\textbf{u}^n) + \frac{1}{\Delta t} \|\phi^{n+1}-\phi^n\|^2,
\end{equation}
which can be iteratively solved by the following secant method in each iteration:
\begin{equation*}%\label{equ:sigma}
	\sigma^{s+1} = \sigma^s -\frac{F(\sigma^s)(\sigma^s-\sigma^{s-1})}{F(\sigma^s)-F(\sigma^{s-1})},
\end{equation*}
with an initial guess of $\sigma^0$ and $\sigma^1$. 
%
%In order to better locate the zero point $\sigma^*$, we have chosen a larger initial guess. We present the algorithm as following.

\begin{remark}
	The definition of $\phi^{n+1}$ in \eqref{equ:e1} is volume-preserving but not bound-preserving, we employ the same approaches in \eqref{equ:average} and \eqref{equ:phi_new} to ensure that it preserves both bound and volume.
\end{remark}

\begin{remark}
	The existence of solution of the nonlinear system \eqref{equ:nonlinear} is difficult to analysis, but the numerical experiments in the follows imply that the secant method can always converge with the initial guesses $\sigma^0$ and $\sigma^1$. 
\end{remark}

\begin{algorithm}[ht!]
	\DontPrintSemicolon
	\KwIn{$\phi^0$: Initial guess, $\epsilon>0$, $\gamma >0$, $\nu$, $E_{min}$, $\beta$, $N_{max}$ be the maximum number of iteration, tol.}
	\KwOut{$\phi^\ast \in \mathcal{H}$.}
	Initialize $n=1$.\\
	\While{$n < N_{max}$~$\&$~~$ |J(\phi^{n+1},\textbf{u}^{n+1}) - J(\phi^n,\textbf{u}^n)| > tol$}{
		\textbf{1. } Compute $\Breve{\phi}^{n+1}$ by Algorithm \ref{a:1}.\\
		\textbf{2. Objective functional decay step.} Set $s=1$, $\sigma^0$, $\sigma^1$.\\
		\While{$J(\phi^{n+1},\textbf{u}^{n+1})>J(\phi^n,\textbf{u}^n)$}{
			Compute $\sigma^{s+1}$ by \[
			\sigma^{s+1} = \sigma^s -\frac{F(\sigma^s)(\sigma^s-\sigma^{s-1})}{F(\sigma^s)-F(\sigma^{s-1})},\]
			
			Compute $\phi^{s+1}$ by
			\[
			\phi^{s+1} = \frac{\Breve{\phi}^{n+1} + \sigma^{s+1}}{\int_\Omega\big( \Breve{\phi}^{n+1} + \sigma^{s+1} \big)d\textbf{x}} V_0.
			\]
			
			{Bound- and volume-preserving is achieved as follows:
				\[
				\text{D}_1 := \{p\in \mathcal{N},\phi^{s+1}(p)=0,~\text{or}~\phi^{s+1}(p)=1 \},
				~~\text{D}_2 := \mathcal{N}~\backslash \text{D}_1,
				\]
				\[\Breve{\phi}^{s+1}=\begin{cases}
					&\phi^{s+1},~~\text{in}~~\text{D}_1,\\
					& \theta(\phi^{s+1} - \bar{\phi}^{s+1}) + \bar{\phi}^{s+1},~~\text{in}~~\text{D}_2,
				\end{cases}\]
				where $\theta = \min\bigg\{\bigg|\frac{1-\bar{\phi}^{s+1}}{\phi_{max} -\bar{\phi}^{s+1} }\bigg|,~~ \bigg|\frac{-\bar{\phi}^{s+1}}{\phi_{min} -\bar{\phi}^{s+1} }\bigg|,~~ 1 \bigg\},~~ \bar{\phi}^{s+1} = \frac{\int_{D_2}\hat{\phi}^{s+1}d\textbf{x}}{|D_2|}.$\\}
			Set $\phi^{n+1} = \phi^{s+1}$. \\
			Solve \eqref{equ:pha_ela} to get $\textbf{u}^{n+1}$, and compute $ J(\phi^{n+1},\textbf{u}^{n+1})$.\\
			Set $s=s+1$.
		}
		Set $n = n+1$.
	}
	\caption{An objective functional decaying scheme for Algorithm \ref{a:1}.}
	\label{a:2}
\end{algorithm}

\section{Numerical experiments}\label{sec:results}
\subsection{Discretization in space}
In this section, we first introduce a fully discrete numerical scheme based on the finite element method. Let $\mathcal{T}_h$ be a family of nondegenerate, quasi-uniform partitions of $\Omega$. These partitions consist of triangles or quadrilaterals when $d=2$, or tetrahedra, prisms, or hexahedra when $d=3$. Let $\mathcal{E}_h$ be the set of all edges($d=2$) or faces($d=3$) of $\mathcal{T}_h$, $h_T$ the diameter of any element $T\in\mathcal{T}_h$. $\mathcal{E}_h^I$ is the set of interior edges or faces for $\mathcal{E}_h$. Let $U_h$ denote the standard finite element space of $d-vectors$ whose components are continuous piecewise linear polynomials,
\begin{equation*}
	U_h := \{ \textbf{v}\in L^2(\Omega)^d:~\textbf{v}|_T \in \mathcal{P}_1(T)^k,~\forall~T\in\mathcal{T}_h\}.
\end{equation*}
Let $T_i,~T_j\in\mathcal{T}_h$ and $e=\partial T_i\cap\partial T_j\in\mathcal{E}_h^I$ with the outward unit normal vector $\textbf{n}_e$ exterior to $T_i$. We denote the average and jump for $\textbf{v}\in U_h$ as follows,
\begin{equation*}
	\{\textbf{v}\} := \frac{1}{2}((\textbf{v}|_{T_i})|_e + (\textbf{v}|_{T_j})|_e),~~[\textbf{v}]:= (\textbf{v}|_{T_i})|_e -  (\textbf{v}|_{T_j})|_e.
\end{equation*}

Next, we introduce the continuous piecewise linear finite element spaces as follows,
\begin{equation*}
	V_h := \{ \psi\in H^1(\Omega): \psi\in\mathcal{P}^1(T),~\forall~T\in\mathcal{T}_h\}.
\end{equation*}
For the solutions of (\ref{equ:dis_ela}) and (\ref{equ:dis1}), we find $\textbf{u}^{n+1}\in U_h$, $\phi^{n+1}\in V_h$ such that
\begin{align*}
	&(\phi^{n+1}, \psi_h) + \Delta t \gamma\epsilon\sum_{T\in\mathcal{T}_h}(\nabla\phi^{n+1},\nabla\psi_h)_T = \sum_{T\in\mathcal{T}_h}\langle-\frac{\Delta t\gamma}{\epsilon}F'(\phi^n) + \textbf{E}'(\phi^n)\varepsilon(\textbf{u}^n):\varepsilon(\textbf{u}^n),\psi_h\rangle_T, \forall~\psi_h\in V_h,\\
	&\mathcal{A}(\textbf{u}^{n+1}_h, \textbf{v}_h) = (\textbf{f},\psi_h) + \langle\textbf{s}\cdot \psi_h\rangle_{\Gamma_T},~~\forall ~\textbf{v}_h\in U^0_h,
\end{align*}
where $\mathcal{A}$ is the bilinear form defined as
\begin{align*}
	\mathcal{A}_s(\textbf{u}^{n+1}_h, \textbf{v}_h) =& \sum_{T\in\mathcal{T}_h}(\textbf{E}(\phi^{n+1})\varepsilon(\textbf{u}^{n+1}),\varepsilon(\textbf{v}_h))_T - \sum_{e\in\mathcal{E}_h^I\cup\Gamma_D}\langle\{\textbf{E}(\phi^{n+1})\varepsilon(\textbf{u}^{n+1})\cdot \textbf{n}_e\},[\textbf{v}_h] \rangle_e\nonumber\\
	&- \sum_{e\in\mathcal{E}_h^I\cup\Gamma_D}\langle\{\textbf{E}(\phi^{n+1})\varepsilon(\textbf{v}_h)\cdot \textbf{n}_e\},[\textbf{u}^{n+1}] \rangle_e + \sum_{e\in\mathcal{E}_h^I}\frac{\theta_1}{h_T}\langle [\textbf{u}^{n+1}],[\textbf{v}_h] \rangle_e.
\end{align*}

\begin{remark}
	To address the locking phenomenon that arises when the Poisson's ratio $\nu$ approaches $0.5$ in 3D or $1$ in 2D from \eqref{equ:lame}, we utilize the discontinuous Galerkin finite element method for solving the elasticity equations. When we set $\nu=0.3$, conforming finite element methods remain a viable alternative for obtaining the solution.
\end{remark}

\subsection{2D examples} \label{sec:2dexample}

We demonstrate the unconditional objective functional decay and robustness of our method through the following five classical benchmark problems in topology optimization. 

{\it Example 1.} [Cantilever Beam Variations]
\begin{itemize}
	\item \textbf{Case 1} (Figure \ref{Exa} (a)): Domain $\Omega=(0,2)\times(0,1)$ with Dirichlet condition: $\mathbf{u}=\mathbf{0}$ on $\{0\}\times[0,1]$, Neumann condition: $\mathbf{s}=(0,-1)^\top$ at $\{2\}\times[0.45,0.55]$ and traction-free elsewhere.
	
	\item   \textbf{Case 2} (Figure \ref{Exa} (b)):  Modified boundary conditions: fixed supports at $\Gamma_T=[0,0.05]\times\{0\}$ and $\Gamma_T=[1.95,2]\times\{0\}$ and traction $\mathbf{s}=(0,-1)^\top$ at $[0.95,1.05]\times\{0\}$.
	
	\item \textbf{Case 3} (Figure \ref{Exa} (c)): Modified from Case 1 with  traction $\mathbf{s}=(0,-1)^\top$ distributed over $\Gamma_T=[1.9,2]\times\{0\}$.
	
\end{itemize}

{\it Example 2.} [Bridge Design] (Figure \ref{Exa} (d))

Domain $\Omega=(0,2)\times(0,1)$ with non-structural mass on $[0,2]\times\{1\}$, fixed supports at $[0,0.05]\times\{0\}$ and $(1.95,2)\times\{0\}$, and body force $\mathbf{f}=(0,-0.1)^\top$ representing gravitational load.

{\it Example 3.} [Curved Domain] (Figure \ref{Exa} (e))

Domain bounded by line segments $\{0\}\times[1,2]$ and $\{3\}\times[-1,-2]$, and two smooth curves, with boundary conditions $\mathbf{u}=\mathbf{0}$ on left arc and $\mathbf{s}=(0,-1)^T$ on $\{3\}\times[-1.9,-2]$. Each of these curves is represented by a cubic B$\acute{e}$zier curve. The upper boundary curve is determined by a set of control and end points, specifically $(0,2), (2.5,1.5), (0.8,-1)$, and $(3,-1)$. Similarly, the lower boundary curve is defined by another set of points, namely $(0,1), (1.5,0.5), (0,-2)$, and $(3,-2)$.

In all examples, the material is assumed to be isotropic with a Young's modulus $E=\frac{100}{91}\approx 1.1$,  Poisson's ration $\nu = \frac{3}{7}\approx 0.43$, $E_{min} = 10^{-4} $, and $p=3$, unless otherwise specified. The stopping criteria is the maximum value of $T$ and the tolerance of the difference in the objective function values between two consecutive steps, the initial guess $\sigma^0 = -0.5$, $\sigma^1 = 0$, and a projection by the threshold $0.5$ is used for the presentation of the results after the iteration stops.

\begin{figure}[htbp]
	\centering
	\begin{minipage}[t]{0.6\textwidth} 
		\centering
		\begin{minipage}[t]{0.48\linewidth}
			\centering
			\includegraphics[width=\linewidth,trim=5cm 12cm 4cm 8cm,clip]{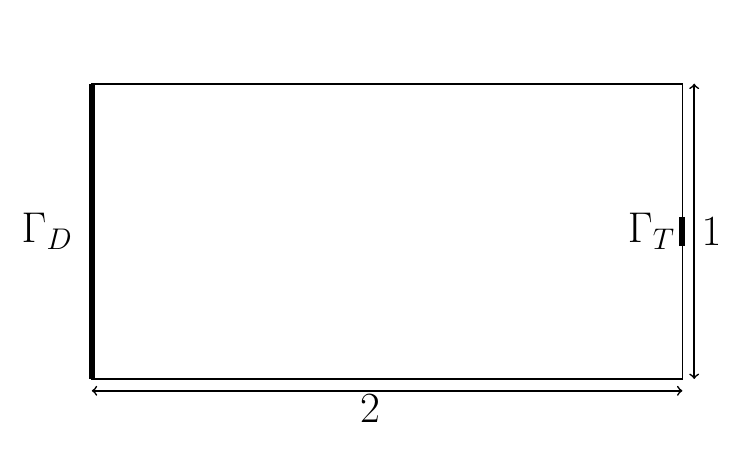}
			\captionsetup{skip=0pt}
			\caption*{(a)}
			\label{fig:m1}
		\end{minipage}
		\hfill
		\begin{minipage}[t]{0.48\linewidth}
			\centering
			\includegraphics[width=\linewidth,trim=5cm 12cm 4cm 8cm,clip]{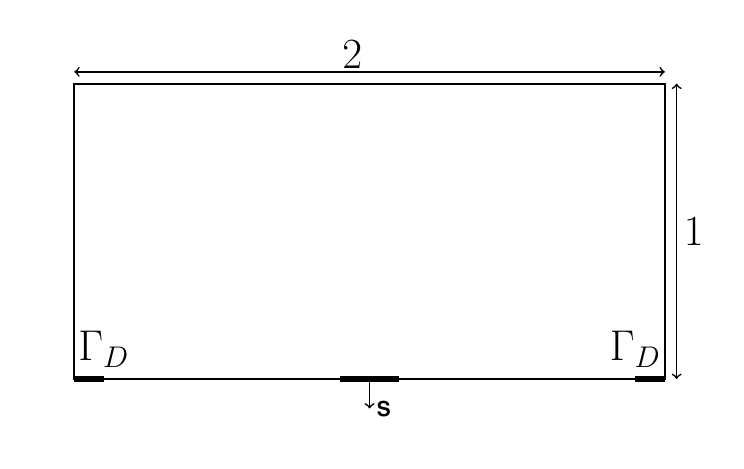}
			\captionsetup{skip=0pt}
			\caption*{(b)}
			\label{fig:m2}
		\end{minipage}
		\par\vspace{0em}
		\begin{minipage}[t]{0.48\linewidth}
			\centering
			\includegraphics[width=\linewidth,trim=5cm 12cm 4cm 8cm,clip]{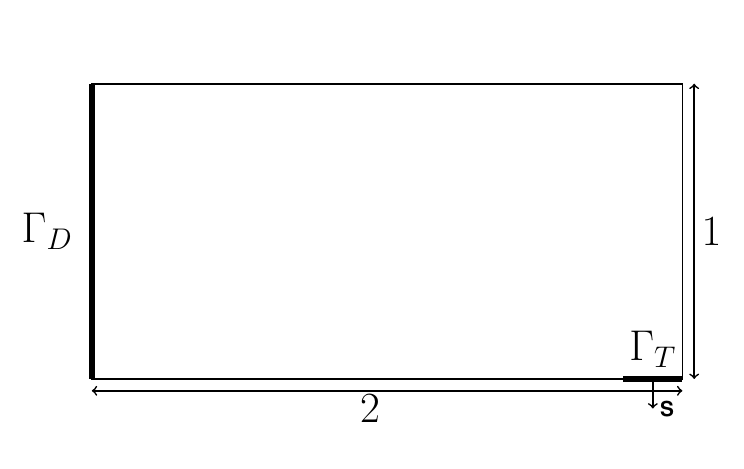}
			\captionsetup{skip=0pt}
			\caption*{(c)}
			\label{fig:m3}
		\end{minipage}
		\hfill
		\begin{minipage}[t]{0.48\linewidth}
			\centering
			\includegraphics[width=\linewidth,trim=5cm 12cm 4cm 8cm,clip]{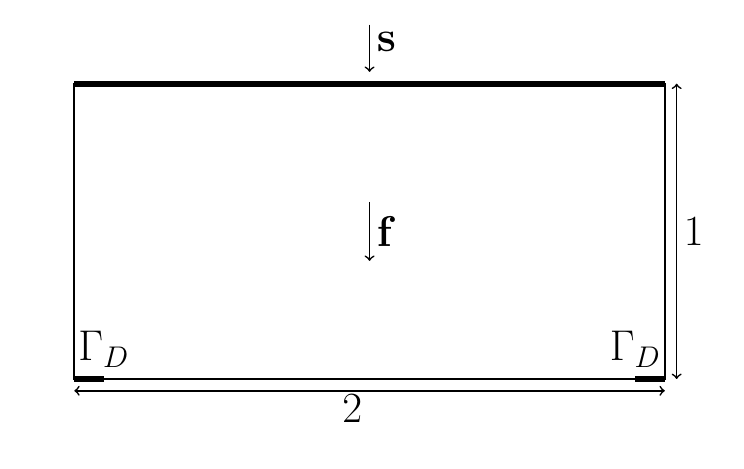}
			\captionsetup{skip=0pt}
			\caption*{(d)}
			\label{fig:m5}
		\end{minipage}
	\end{minipage}
	\hfill
	\begin{minipage}[t]{0.36\textwidth} 
		\centering
		\vspace{-75pt}
		\includegraphics[width=\linewidth,trim=5cm 19cm 12cm 4.1cm,clip]{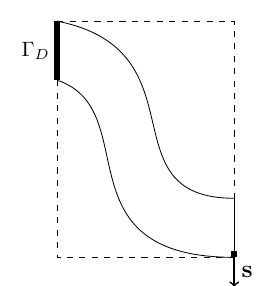}
		\captionsetup{skip=0pt}
		\caption*{(e)}
		\label{fig:m4}
	\end{minipage}
	\caption{Schematic illustration of the geometric structure, loading, and boundary conditions. (a) A cantilever beam with force at the middle of right edge \cite{YU2021126267,JIN2024112932}. (b) A cantilever beam with force at the corner \cite{CiCP-33-4}. (c) A cantilever beam with force at the middle of bottom edge and the two corners of the bottom edge being fixed \cite{YU2021126267,JIN2024112932}. (d) Bridge structure \cite{bruyneel2005note,XU201338}. (e) The curved domain with force in the bottom of the right edge \cite{YU2021126267,JIN2024112932}. See Section~\ref{sec:2dexample}.}\label{Exa}
\end{figure}

\subsubsection{Properties of Algorithm \ref{a:2}} \label{sec:property}

%\subsubsection{Case 1 in Cantilever beam}
%We study the effectiveness and robustness of Algorithm \ref{a:2}, using the boundary conditions in Case 1 of Example 1 illustrated in Figure \ref{Exa} (a). In this example, we set $\Delta t=0.06$, $\gamma = 0.2$, $\epsilon = 0.01$, $\beta = 0.4$, and $T=6$. The computational domain is discretized using a mesh of $400\times 200$.  \DW{iteration stops at the maximal steps?}

We investigate the effectiveness and robustness of Algorithm \ref{a:2} by employing the boundary conditions specified in Case 1 of Example 1, as depicted in Figure \ref{Exa} (a). The computational domain is discretized using a mesh of $400\times 200$.

{\bf The objective functional decaying property.} We begin by examining the objective functional decay properties of Algorithm~\ref{a:2}. Figure~\ref{fig:evo1} compares the evolution of the objective functional $\mathcal{J}(\phi,\mathbf{u})$ for Algorithms~\ref{a:1} and~\ref{a:2}, using a uniform initial distribution $\phi^0(\mathbf{x}) \equiv \beta$ with $\Delta t=0.06$, $\gamma = 0.2$, $\epsilon = 0.01$, $\beta = 0.4$, and $T=6$. This comparison reveals three key distinctions: 1. Algorithm~\ref{a:2} exhibits a strictly monotonic decrease in the compliance functional, achieving a final value of $\mathcal{J}_{\text{final}} = 0.98$---lower than the $1.08$ attained by Algorithm~\ref{a:1}; 2. The optimized material distribution from Algorithm~\ref{a:2} displays more intricate load-bearing structures, with additional major branches evident in the resulting topology; 3. During optimization, Algorithm~\ref{a:2} automatically activated its objective functional correction mechanism multiple times, ensuring monotonic decay throughout the iteration process.

In addition, the step size sensitivity study reveals important stability characteristics. When reducing $\Delta t$ to $0.05$, both algorithms converge to similar topological configurations as displayed in Figure~\ref{fig:evo2}. However, Algorithm~\ref{a:1} exhibits persistent oscillations in the objective functional around steady state. In contrast, Algorithm~\ref{a:2} maintains strict monotonic decay.

\begin{figure}[ht!]
	\centering
	\includegraphics[width=0.45\linewidth,trim=2cm 1.5cm 3cm 2cm,clip]{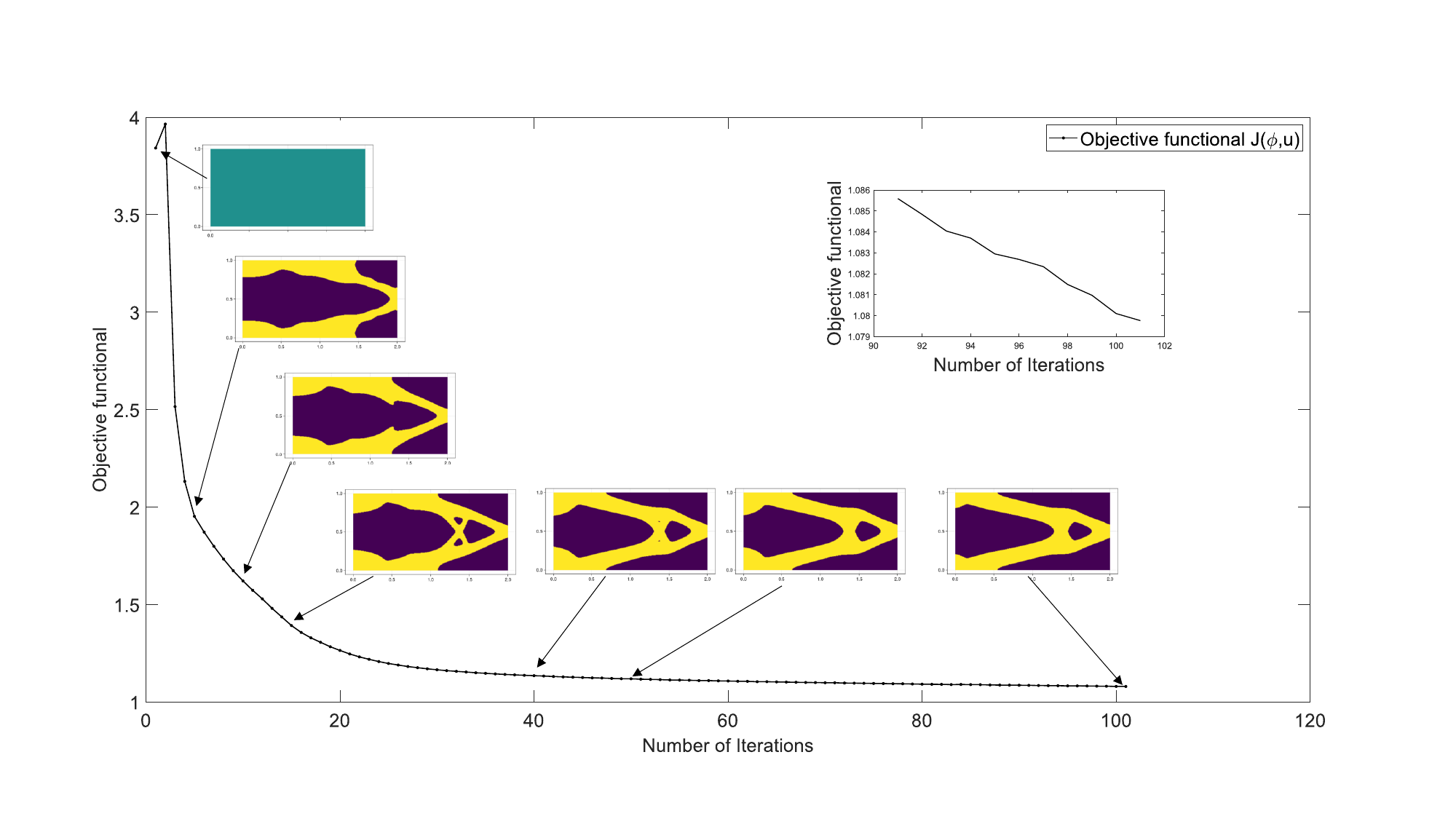}
	\includegraphics[width=0.45\linewidth,trim=2cm 1.5cm 3cm 2cm,clip]{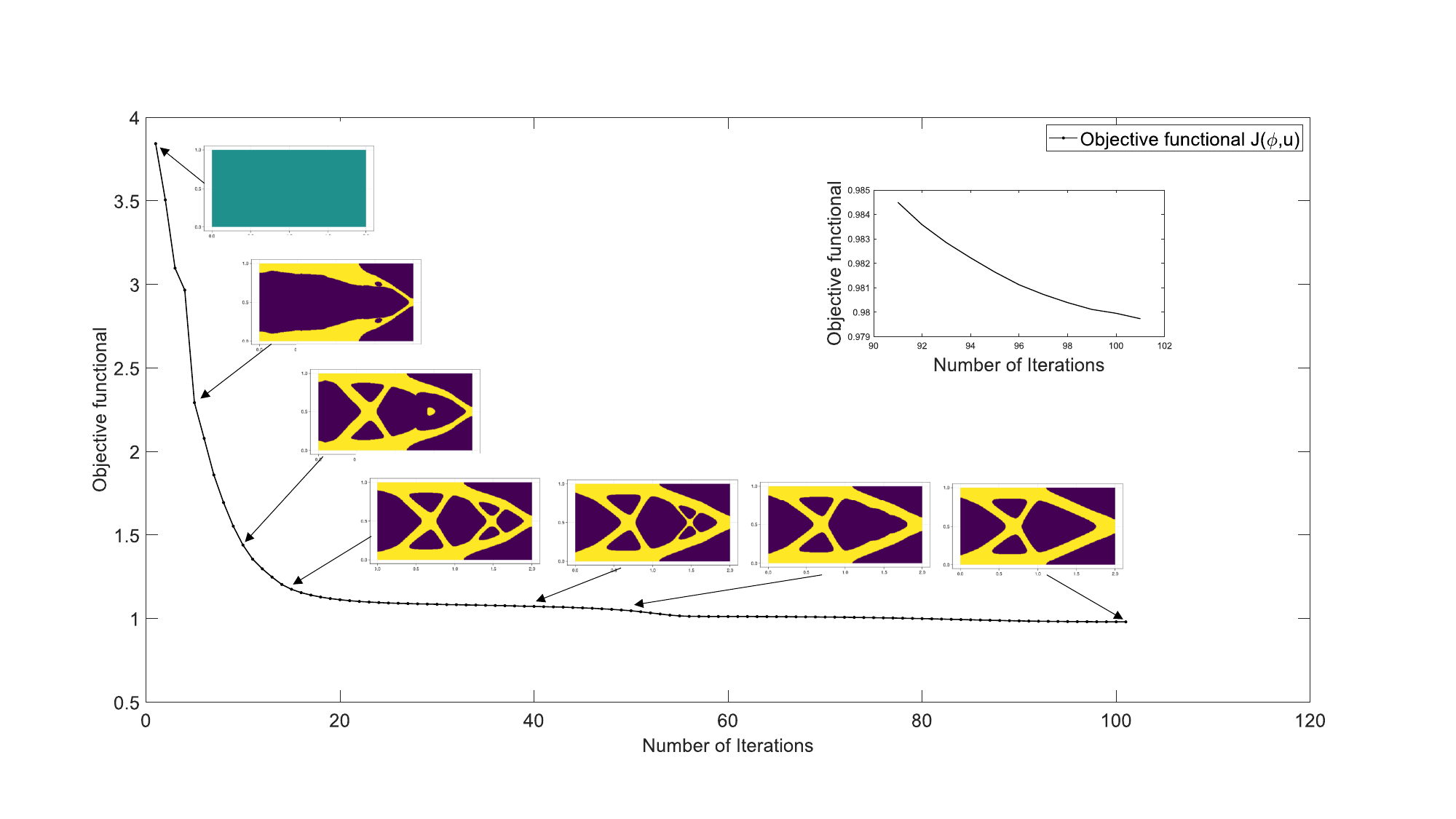}
	\caption{Evolution of the approximate solutions $\phi$ and the objective functional values during iterations with $\Delta t =0.06$ using Algorithm \ref{a:1} (left) and Algorithm \ref{a:2} (right). See Section~\ref{sec:property}.}
	\label{fig:evo1}
\end{figure}

\begin{figure}[ht!]
	\centering
	\includegraphics[width=0.45\linewidth,trim=2cm 1.5cm 3cm 2cm,clip]{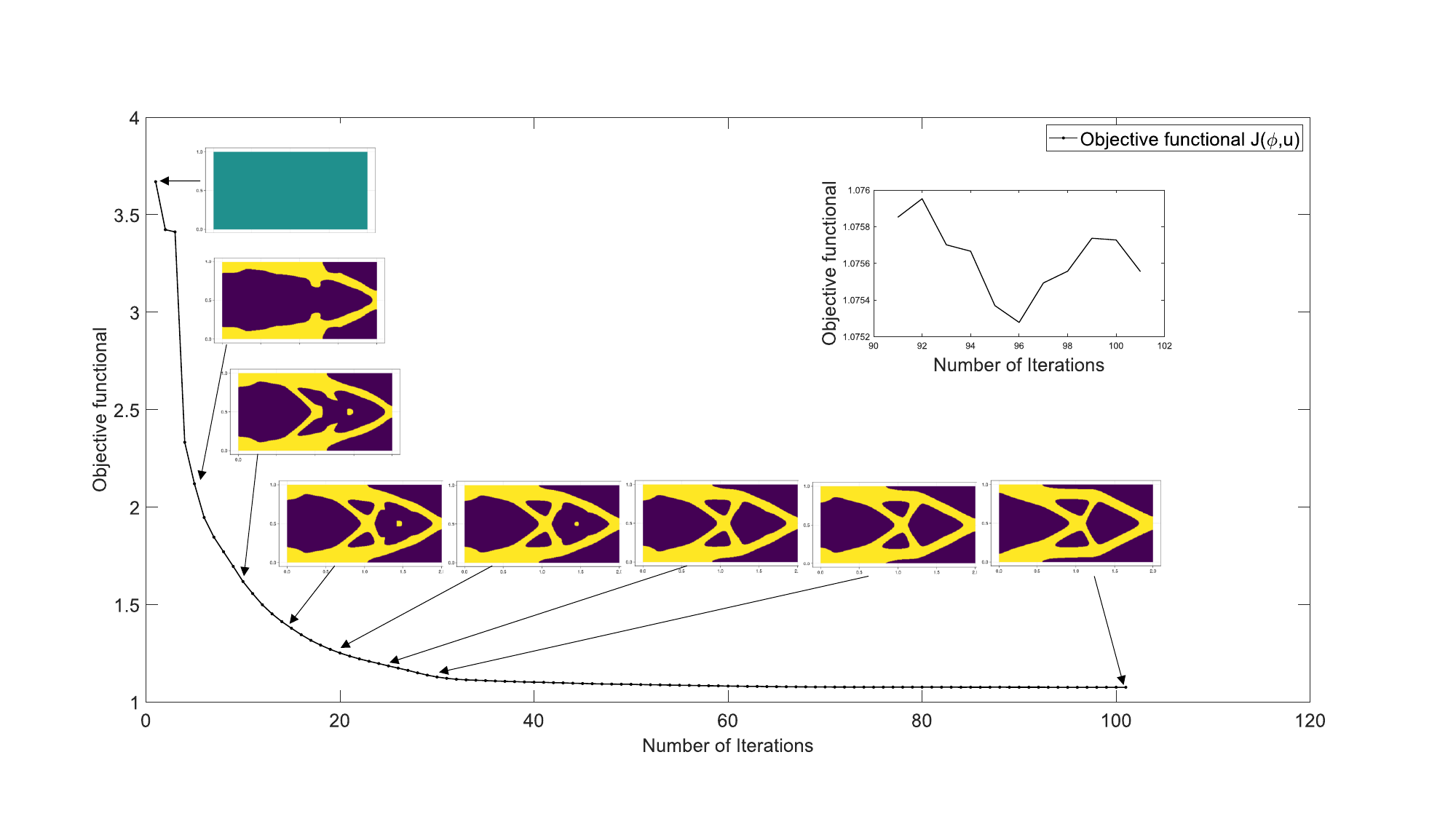}
	\includegraphics[width=0.45\linewidth,trim=2cm 1.5cm 3cm 2cm,clip]{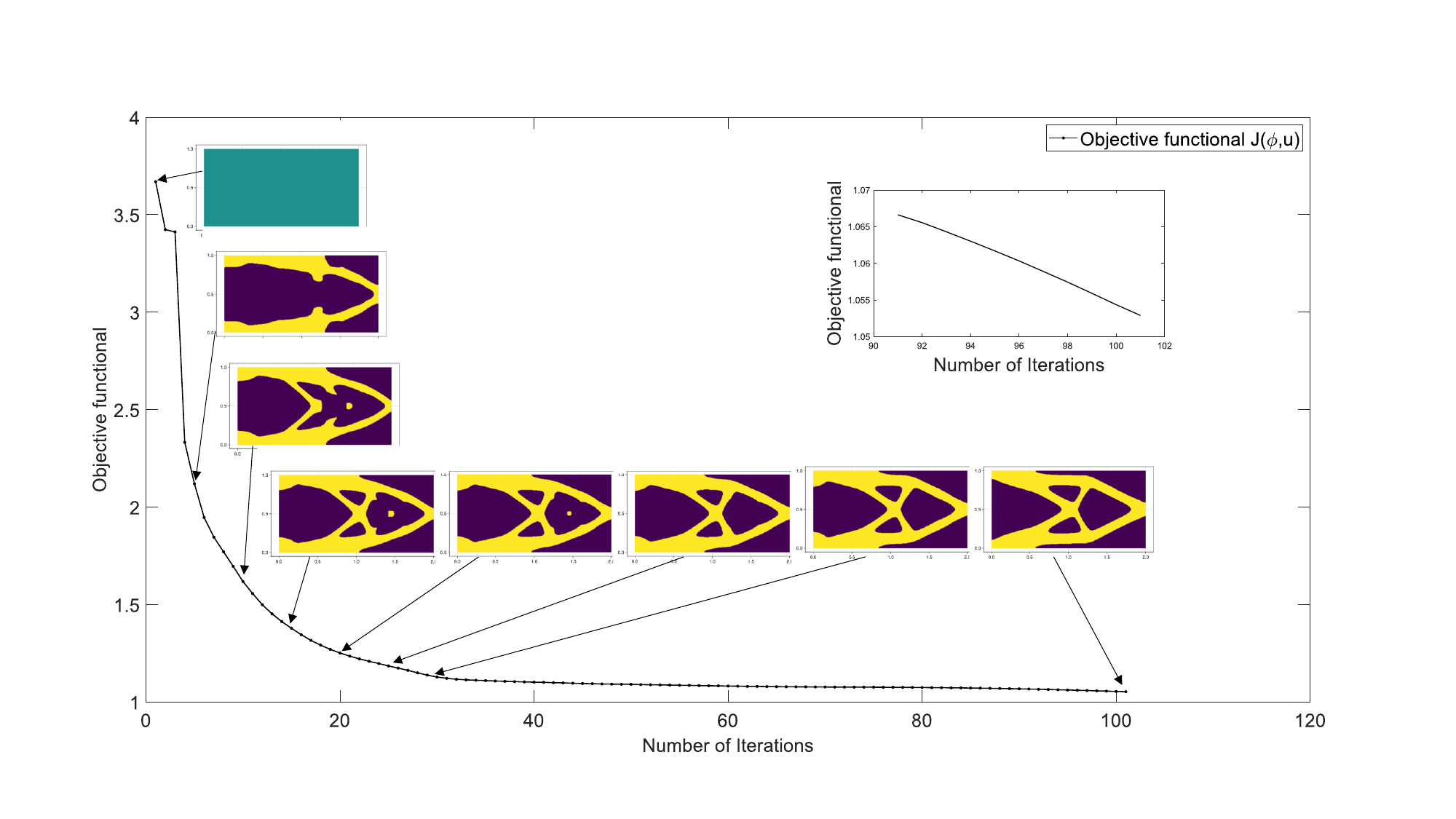}
	\caption{Evolution of the approximate solutions $\phi$ and the objective functional values during iterations with $\Delta t =0.05$ using Algorithm \ref{a:1} (left) and Algorithm \ref{a:2} (right). See Section~\ref{sec:property}.}
	\label{fig:evo2}
\end{figure}

{\bf Bound- and volume-preserving.} Using the same test with $\Delta t=0.05$, $\gamma = 0.2$, $\epsilon = 0.01$, $\beta = 0.4$, and $T=5$, we quantitatively verify the constraint-preserving properties of Algorithm~\ref{a:2}. Figure~\ref{fig:evo_volume} demonstrates that the volume fraction remains strictly conserved throughout all iterations, while the phase field function maintains its prescribed bounds ($\phi_{\min} \leq \phi \leq \phi_{\max}$) without violation. These results confirm Algorithm~\ref{a:2} successfully enforces all constraints during optimization.

\begin{figure}[htbp]
	\centering
	\includegraphics[width=0.4\linewidth,trim=1.7cm 8.7cm 1.7cm 8.5cm,clip]{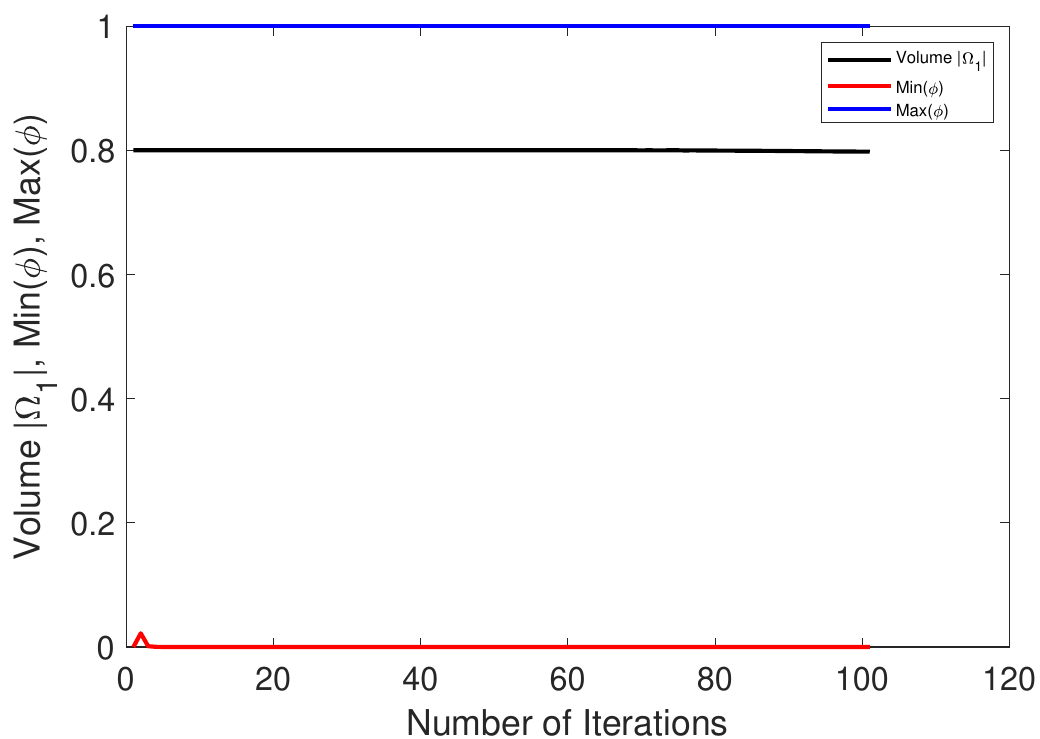}
	\caption{The volume $|\Omega_1|$, $\max\{\phi\}$ and $\min\{\phi\}$ with constant initial distribution computed by Algorithm \ref{a:2}. See Section~\ref{sec:property}.}
	\label{fig:evo_volume}
\end{figure}

{\bf Poisson's ratio approaches $1$.} 
We further investigate the performance as the Poisson's ratio approaches the incompressible limit ($\nu \to 1$) in 2D. The results demonstrate that the discontinuous Galerkin finite element method effectively prevents volumetric locking while Algorithm~\ref{a:2} maintains stable evolution of the phase field $\phi$. 

Figure~\ref{fig:possion} shows the converged material distribution and objective functional values for parameters $\nu = 0.96$, $E = 1.32$, $\gamma = 0.1$, $\beta = 0.3$, $\epsilon = 0.01$, and $\Delta t = 0.01$ on a $400 \times 200$ grid with uniform random initialization. The solution exhibits stable convergence of $\phi$ with monotonic decrease of the objective functional throughout all iterations.

\begin{figure}[htbp]
	\centering
	\includegraphics[width=0.45\linewidth,trim=1.2cm 18cm 4cm 3cm,clip]{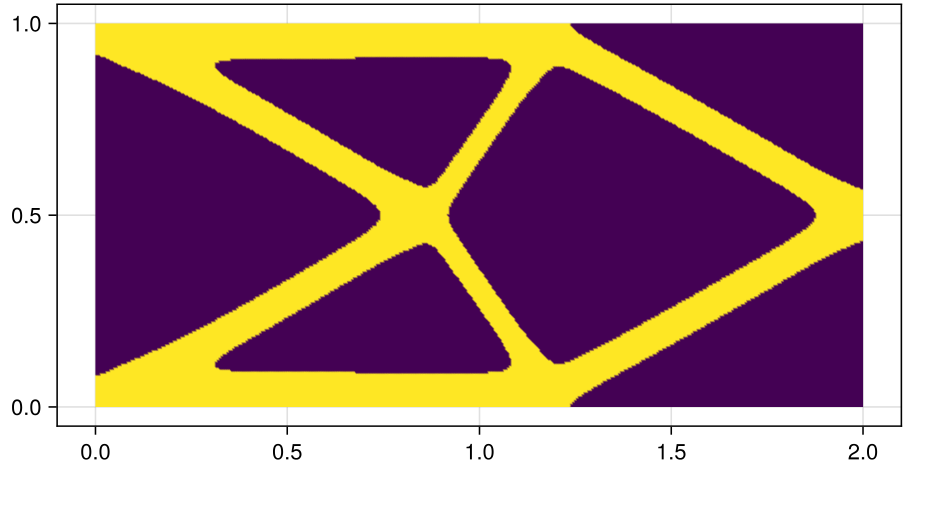}
	\includegraphics[width=0.45\linewidth,trim=1cm 8.5cm 3cm 9.8cm,clip]{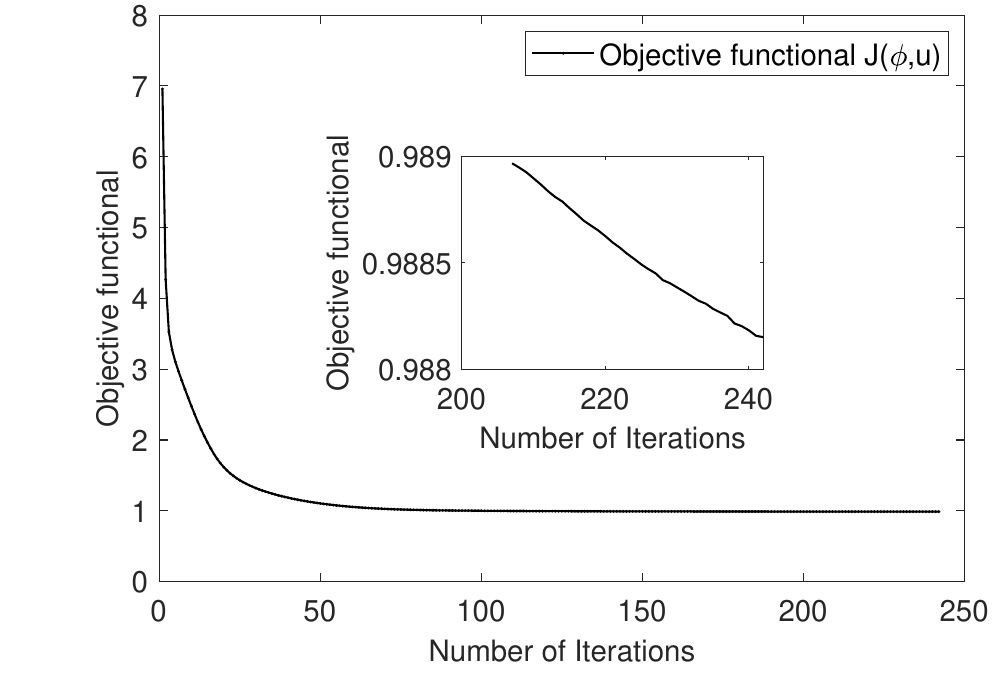}
	\caption{The approximate optimal solutions of $\phi$ and the objective functional decaying curve with Possion's ration $\nu = 0.96$ and Young's modulus $E=1.32$. See Section~\ref{sec:property}.}
	\label{fig:possion}
\end{figure}

{\bf Effect of the mesh size.} 
Figure~\ref{fig:E1_mesh} presents the optimized material distributions $\phi$ and corresponding objective functional decay across various grid resolutions. The results demonstrate excellent stability of the $\phi$ solutions under mesh refinement, with the objective functional maintaining consistent decay profiles regardless of grid size. This robust behavior confirms the algorithm's mesh independence, as both solution quality and convergence characteristics remain unaffected by discretization changes.

\begin{figure}[htbp]
	\centering
	\includegraphics[width=0.3\linewidth,trim=1.2cm 18cm 4cm 3cm,clip]{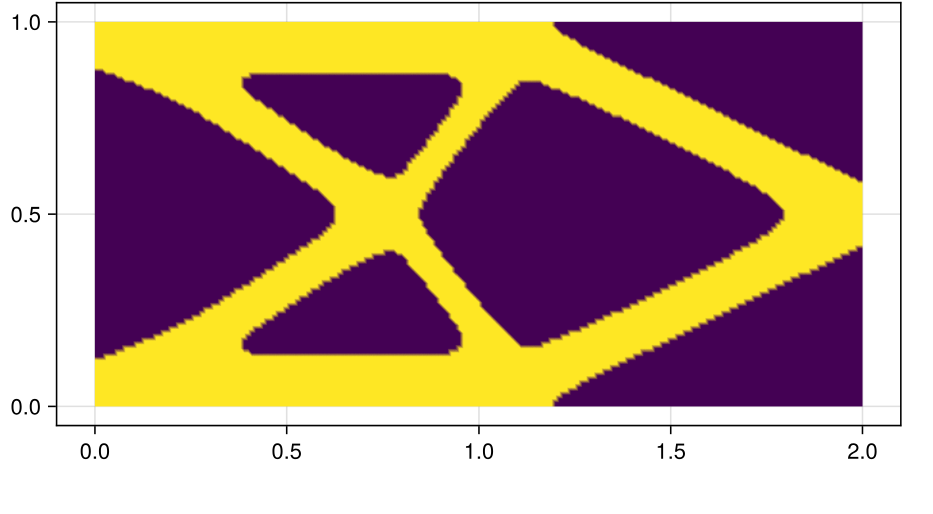}
	\includegraphics[width=0.3\linewidth,trim=1.2cm 18cm 4cm 3cm,clip]{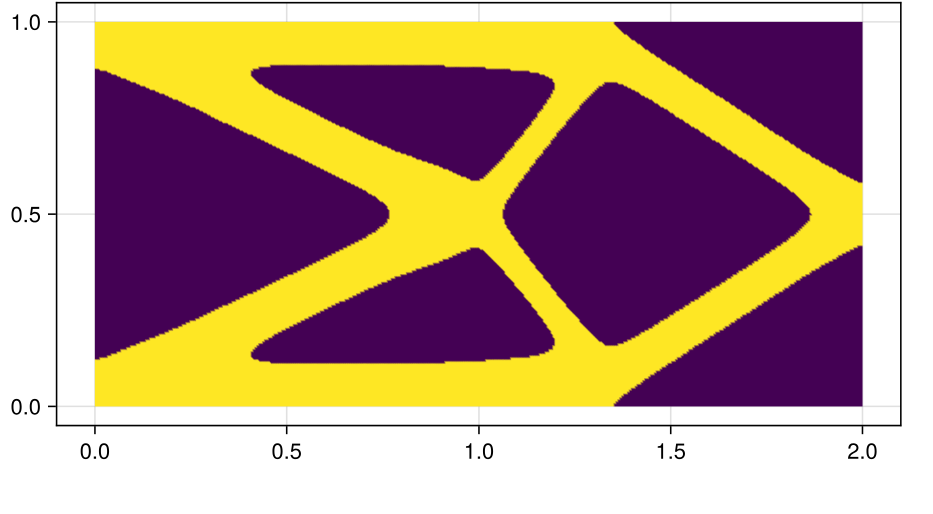}
	\includegraphics[width=0.3\linewidth,trim=1.2cm 18cm 4cm 3cm,clip]{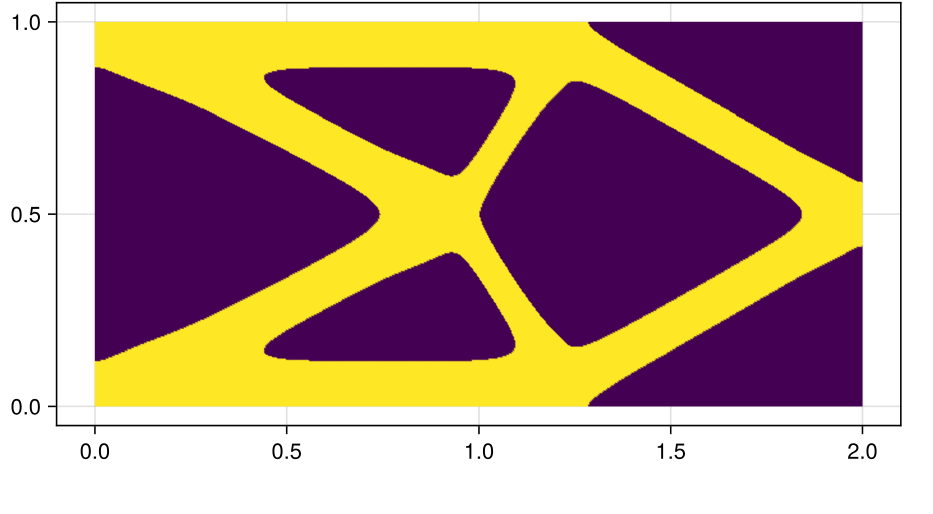}\\
	\includegraphics[width=0.3\linewidth,trim=2cm 8.5cm 3cm 9.8cm,clip]{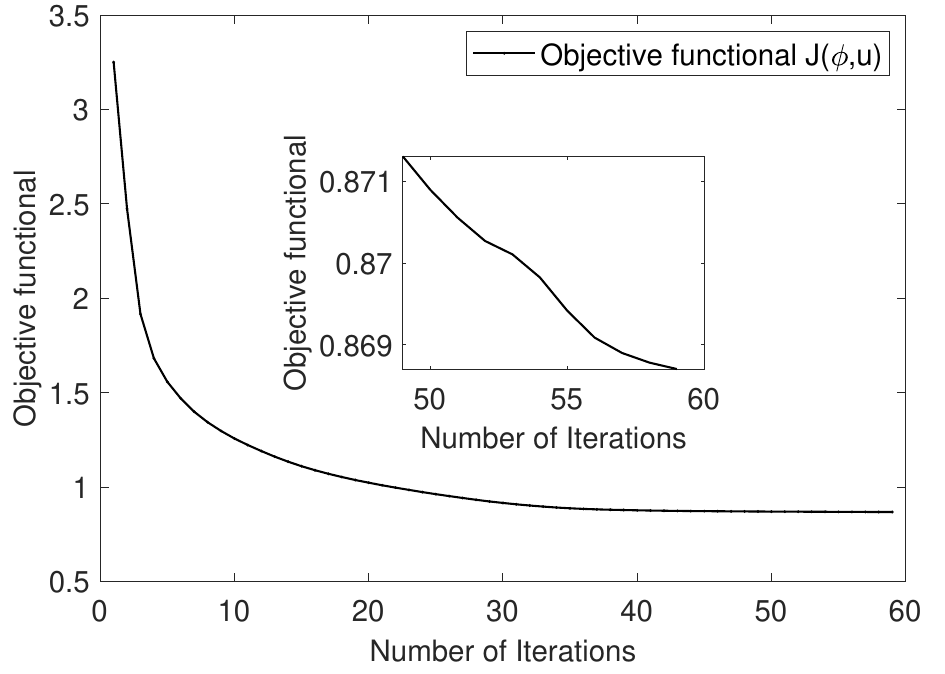}
	\includegraphics[width=0.3\linewidth,trim=2cm 8.5cm 3cm 9.8cm,clip]{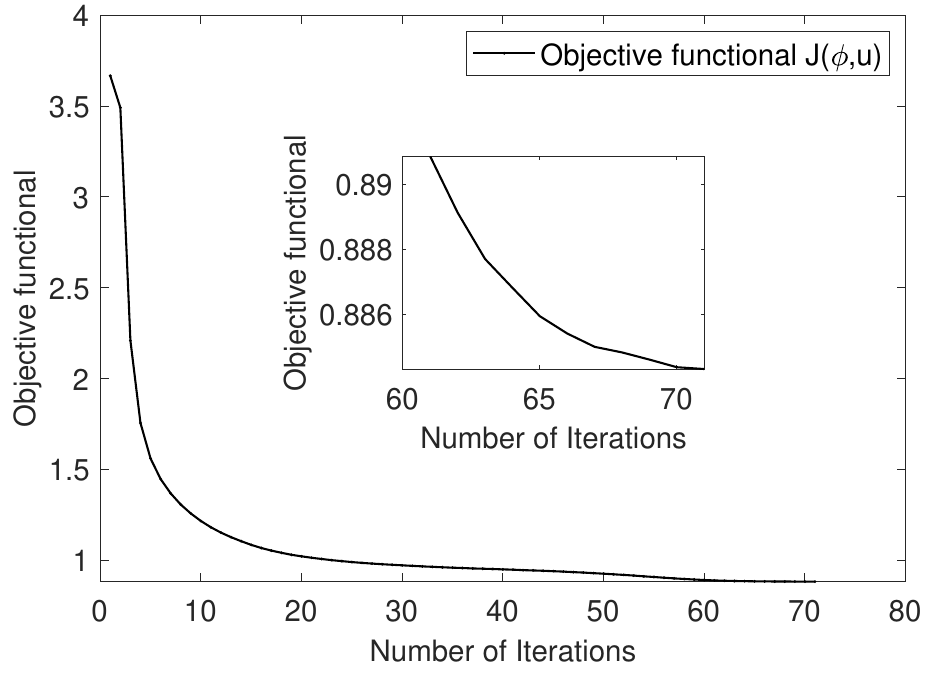}
	\includegraphics[width=0.3\linewidth,trim=2cm 8.5cm 3cm 9.8cm,clip]{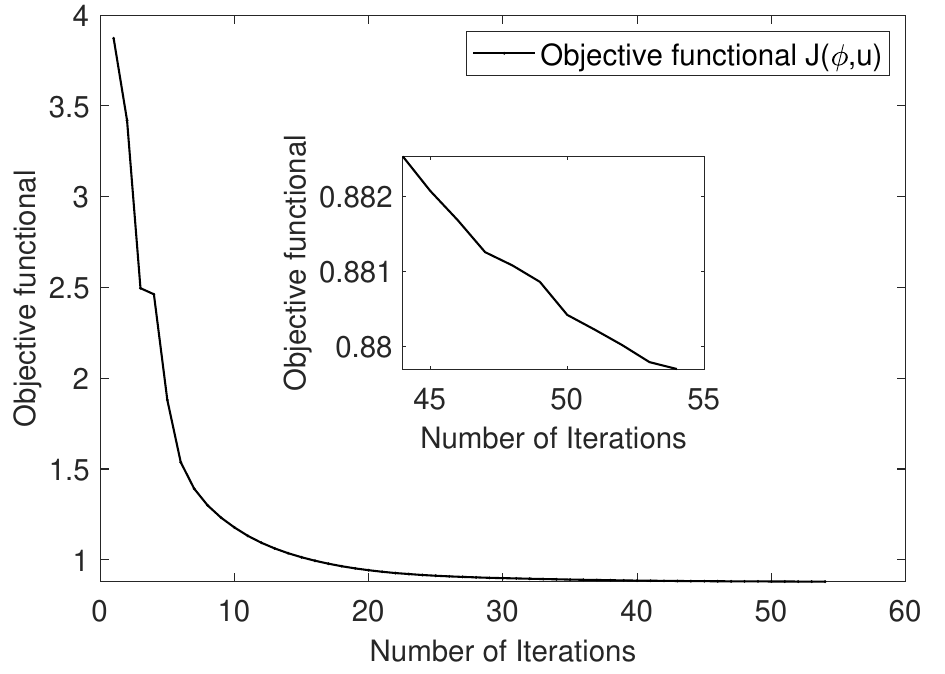}
	\caption{Effects of mesh on the approximate optimal solutions of $\phi$ and objective functional values with $\gamma = 0.1$, $\beta=0.4$, $\epsilon = 0.01$, $\Delta t = 0.05$, $T=5$ and a uniform random initial distribution of $\phi$. From left to right,  $200\times 100, 400\times 200, 600\times 300$ grids are used, respectively. See Section~\ref{sec:property}.}
	\label{fig:E1_mesh}
\end{figure}

\subsubsection{Profile dependency on parameters} \label{sec:profile}

In this section, we investigate the dependence of the optimal profile on the parameters ($\gamma$, $\beta$, $\epsilon$) using the cantilever problem illustrated in Figure~\ref{Exa}(b). All simulations employ fixed parameters $\Delta t = 0.01$, $T = 5$, and a $400 \times 200$ computational mesh.
	
{\bf Effect of the weighting parameter $\gamma$.} Figure~\ref{fig:E3_gamma} presents the approximate optimal solutions for $\phi$ with $\gamma = 0.05, 0.01, 0.005$. The results demonstrate that smaller values of $\gamma$ produce finer structural details in the optimized profile, confirming its role as a geometric resolution control parameter.
	
{\bf Effect of volume fraction $\beta$.} Figure~\ref{fig:E3_volume} shows the optimal $\phi$ solutions and corresponding objective functional values for $\beta = 0.1, 0.2, 0.3$. We observe that while larger $\beta$ values yield thicker structural members, the essential topological features remain qualitatively similar. This suggests our algorithm robustly preserves the characteristic design patterns across different material constraints, with $\beta$ mainly influencing structural scale rather than topological configuration.

{\bf Effect of the interface thickness $\epsilon$.} The interface thickness parameter $\epsilon$ critically governs phase field evolution dynamics. As demonstrated in Figure~\ref{fig:E3_epsilon}, smaller $\epsilon$ values ($\epsilon \to 0$) generate sharper material interfaces and increased hole density, while larger values produce smoother transitions. Notably, the optimization process maintains excellent convergence properties across all $\epsilon$ values, confirming that this parameter primarily controls geometric refinement without affecting solution feasibility.
	
	\begin{figure}[htbp]
		\centering
		\includegraphics[width=0.3\linewidth,trim=1.2cm 18cm 4cm 3cm,clip]{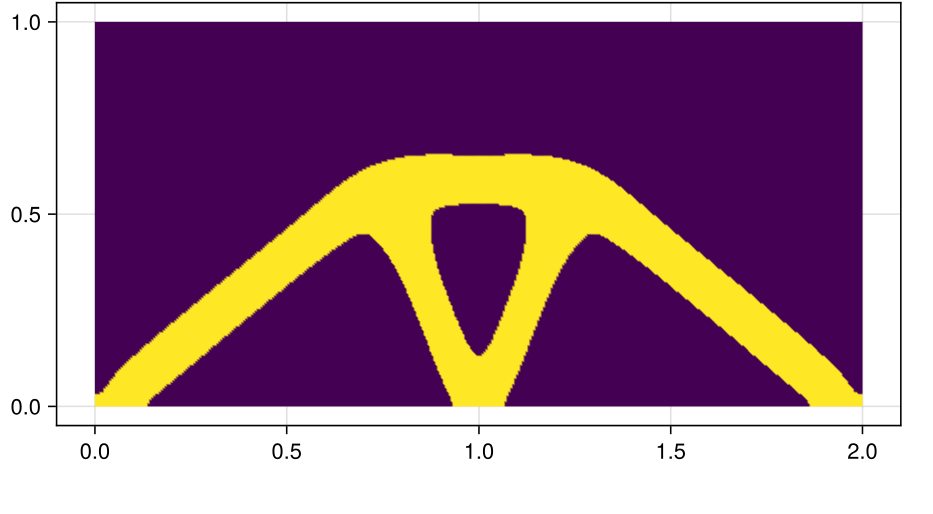}
		\includegraphics[width=0.3\linewidth,trim=1.2cm 18cm 4cm 3cm,clip]{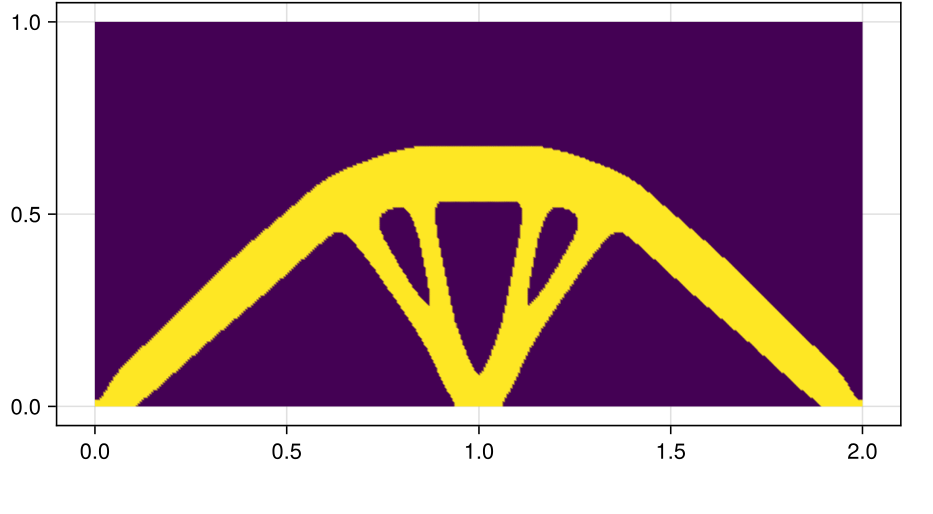}
		\includegraphics[width=0.3\linewidth,trim=1.2cm 18cm 4cm 3cm,clip]{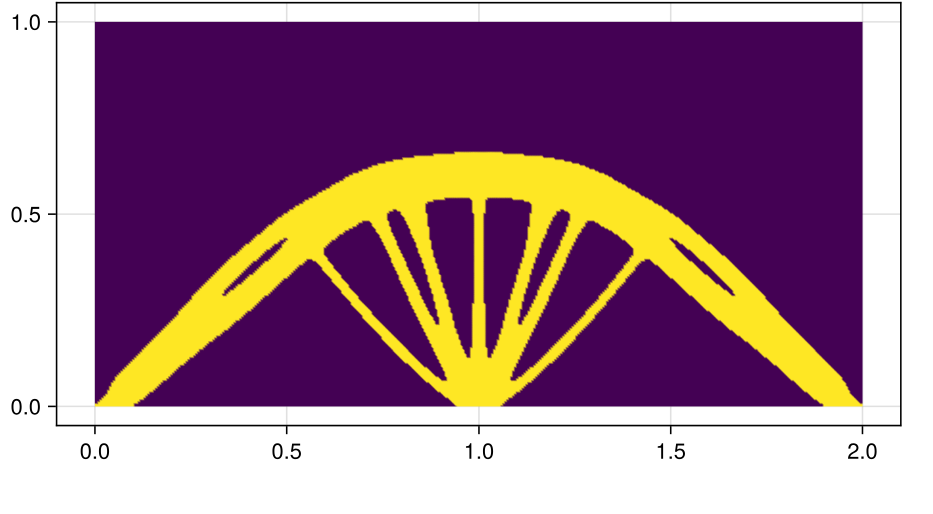}
		\caption{Effects of $\gamma$ on the approximate optimal solutions of $\phi$ on a $400\times 200$ grid with $\beta=0.2$, $\epsilon = 0.01$, $\Delta t = 0.01$, $T=5$ and and a uniform random initial distribution of $\phi$. From left to right, $\gamma = 0.05,~0.01,~0.005$ are used, respectively.  See Section~\ref{sec:profile}.}
		\label{fig:E3_gamma}
	\end{figure}
	
	\begin{figure}[htbp]
		\centering
		\includegraphics[width=0.3\linewidth,trim=1.2cm 18cm 4cm 3cm,clip]{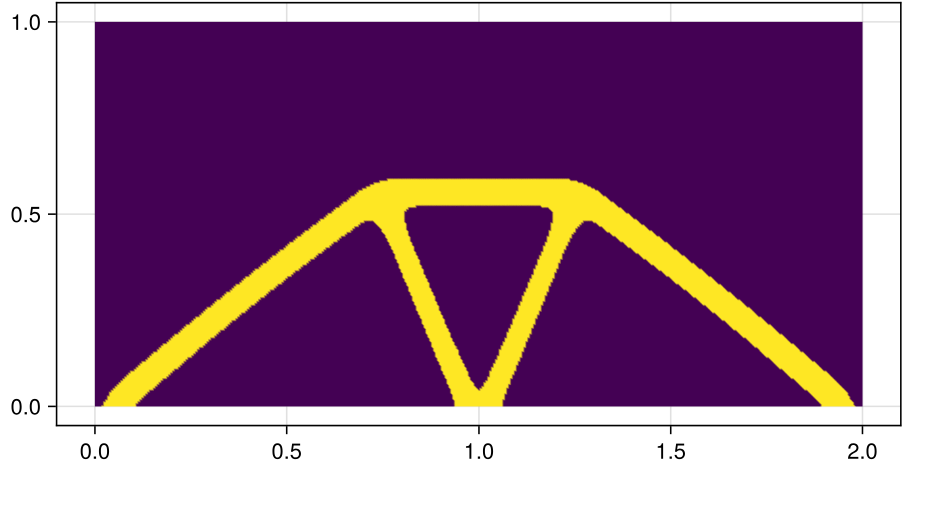}
		\includegraphics[width=0.3\linewidth,trim=1.2cm 18cm 4cm 3cm,clip]{pic/m3/gamma/gamma5.pdf}
		\includegraphics[width=0.3\linewidth,trim=1.2cm 18cm 4cm 3cm,clip]{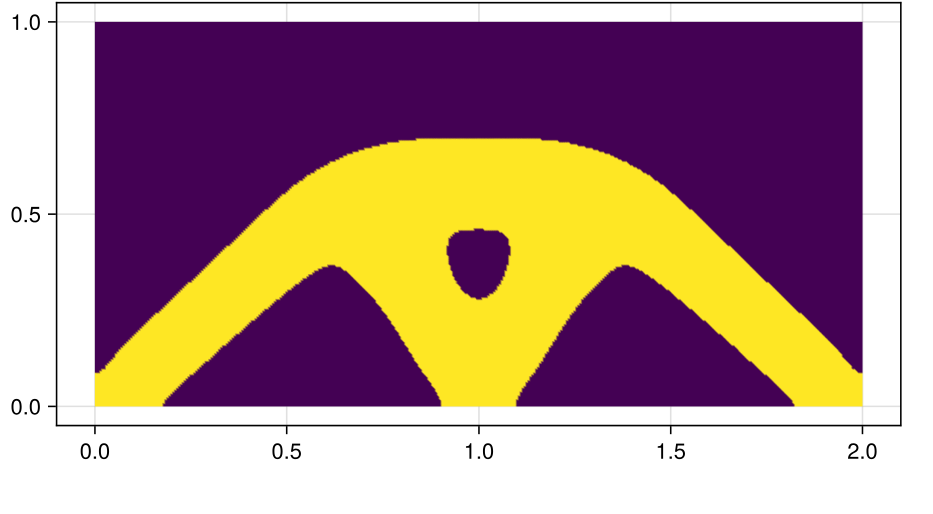}
		\caption{Effects of volume on the approximate optimal solutions of $\phi$ on a $400\times 200$ grid with $\gamma=0.05$, $\epsilon = 0.01$, $\Delta t = 0.01$, $T=5$ and a uniform random initial distribution of $\phi$. From left to right, $\beta = 0.1,~0.2,~0.3$ are used, respectively. See Section~\ref{sec:profile}.}
		\label{fig:E3_volume}
	\end{figure}
	
	\begin{figure}[htbp]
		\centering
		\includegraphics[width=0.3\linewidth,trim=1.2cm 18cm 4cm 3cm,clip]{pic/m3/gamma/gamma5.pdf}
		\includegraphics[width=0.3\linewidth,trim=1.2cm 18cm 4cm 3cm,clip]{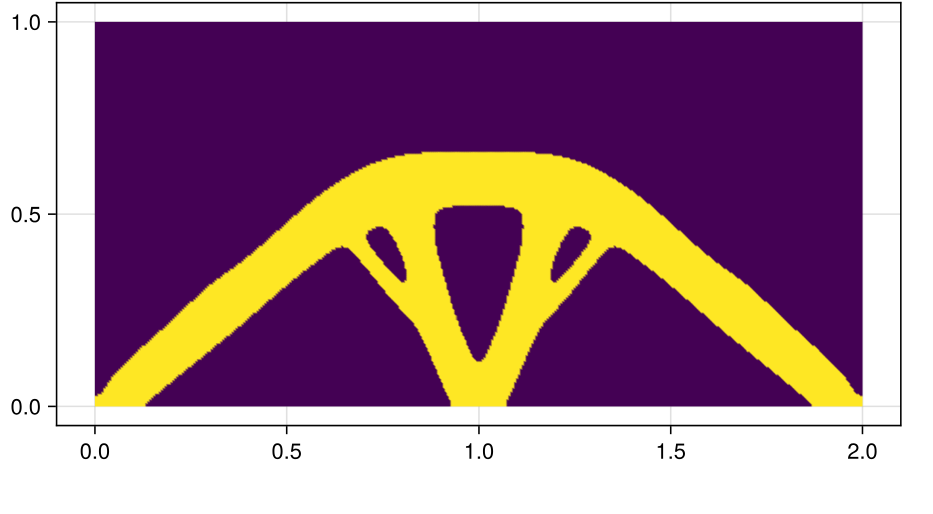}
		\includegraphics[width=0.3\linewidth,trim=1.2cm 18cm 4cm 3cm,clip]{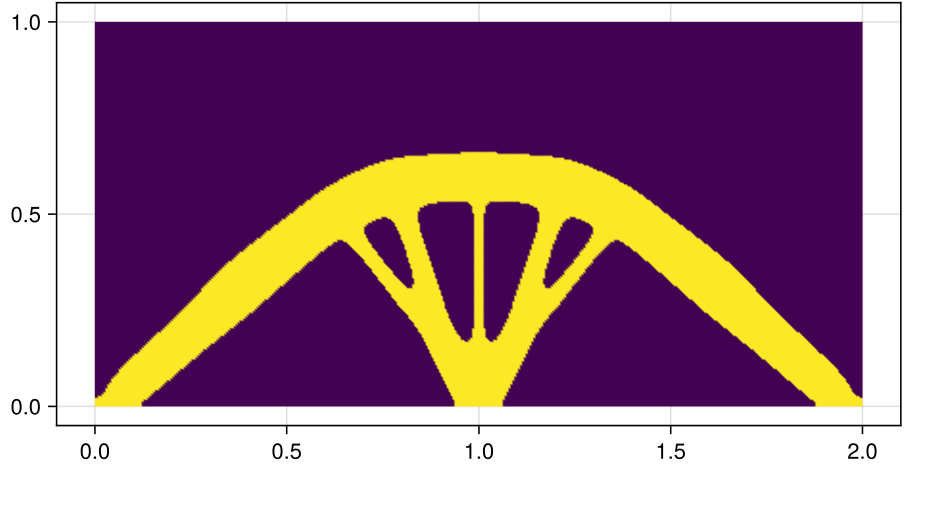}
		\caption{Effects of $\epsilon$ on the approximate optimal solutions of $\phi$ on a $400\times 200$ grid with $\gamma=0.05$, $\beta = 0.2$, $\Delta t = 0.01$, $T=5$ and a uniform random initial distribution of $\phi$. From left to right, $\epsilon = 0.01, ~0.005,~0.003$ are used, respectively. See Section~\ref{sec:profile}.}
		\label{fig:E3_epsilon}
	\end{figure}

\subsubsection{More classical benchmark problems}\label{sec:more}
	In this section, we evaluate our approach on additional classical problems, as illustrated in Figure \ref{Exa}. Figure \ref{fig:E2_volume} presents the approximate optimal solutions for the variable $\phi$ under the following parameters: $\gamma = 0.2$, $\epsilon = 0.01$ (left) and $0.025$ (right), $\Delta t = 0.01$, $\beta = 0.4$, and $T = 1$. The simulations were performed on a $400 \times 200$ grid with a constant initial distribution of $\phi^0 = 0.8$.
	
	\begin{figure}[htbp]
		\centering
		\includegraphics[width=0.3\linewidth,trim=1.2cm 18cm 4cm 3cm,clip]{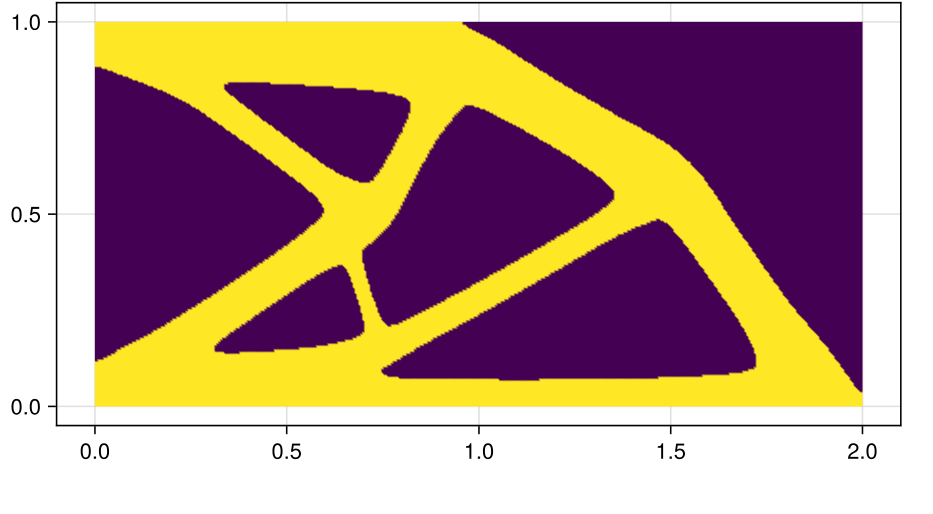}\ \ \ \ \ \ \ \ \ \ 
		\includegraphics[width=0.3\linewidth,trim=1.2cm 18cm 4cm 3cm,clip]{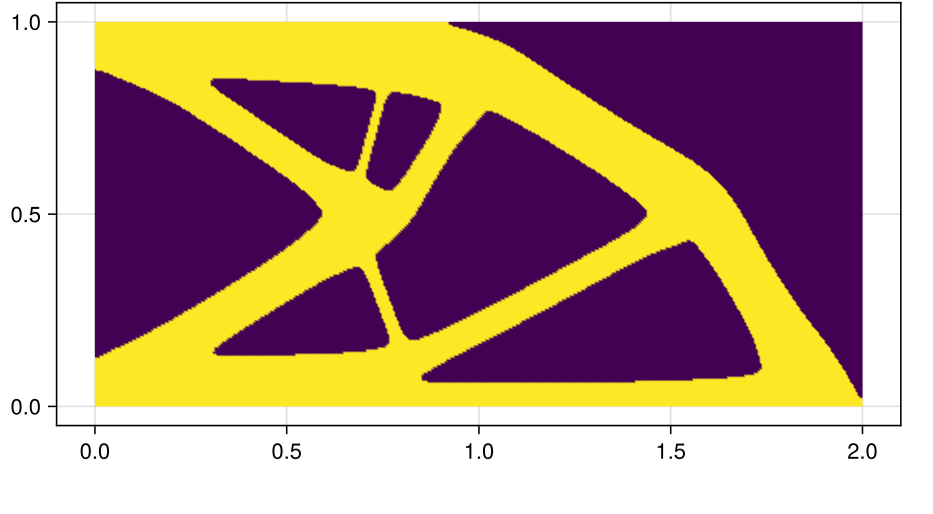}
		\caption{The approximate optimal solutions of $\phi$ with $\gamma=0.2$, $\beta = 0.4$, $\Delta t = 0.01$, $T=1$, $\epsilon = 0.01$ (left) and $0.025$ (right), and same constant initial distribution $\phi^0 = 0.8$ on a $400\times200$ grid. See Section~\ref{sec:more}.}
		\label{fig:E2_volume}
	\end{figure}

We investigate the structural optimization of a bridge configuration, as shown in Figure \ref{Exa}(d). The simulation incorporates the roadway weight effect through a non-structural distributed load applied at the bridge deck level. The initial distribution of $\phi$ is illustrated in Figure \ref{self_weight}, with the parameters set to $\Delta t = 0.002$ and $T = 1$. 

Figure \ref{fig:E5_bodyforce} demonstrates how varying traction forces affect the optimal solutions for $\phi$. As the traction increases, we observe progressively more branched solution patterns. This branching behavior indicates that higher traction forces lead to both greater morphological complexity in the solutions and increased challenges in the optimization convergence. The results clearly show that mechanical loading conditions play a critical role in determining both the solution characteristics and the computational behavior of the structural optimization process. As the surface tractions $\textbf{s}$ increases, the value of the objective functional after stabilization correspondingly increases.

%\DW{maybe need to add the energy decaying curves here to illustrate the convergence.}

	\begin{figure}[htbp]
		\centering
		\includegraphics[width=0.3\linewidth,trim=1.2cm 18cm 4cm 3cm,clip]{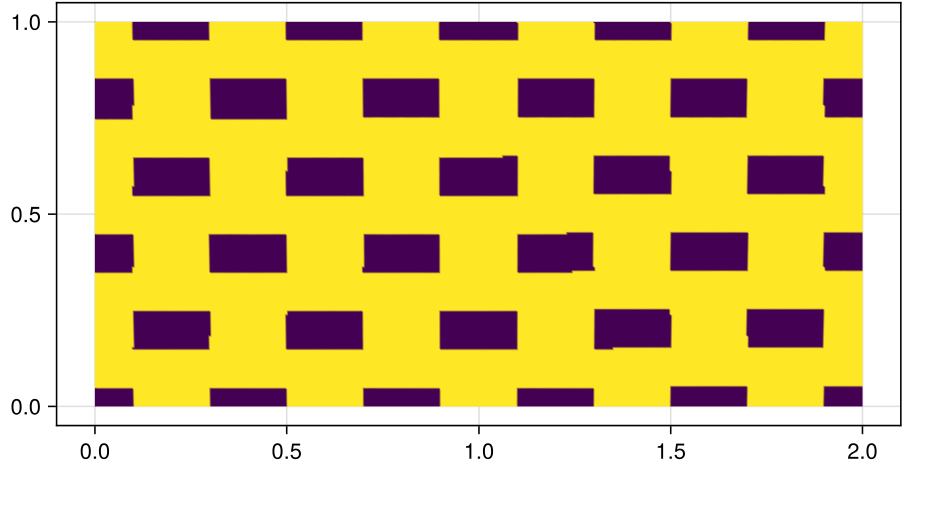}
		\caption{The initial distribution of $\phi$. See Section~\ref{sec:more}.}
		\label{self_weight}
	\end{figure}

	\begin{figure}[htbp]
		\centering
		\includegraphics[width=0.3\linewidth,trim=1.2cm 18cm 4cm 3cm,clip]{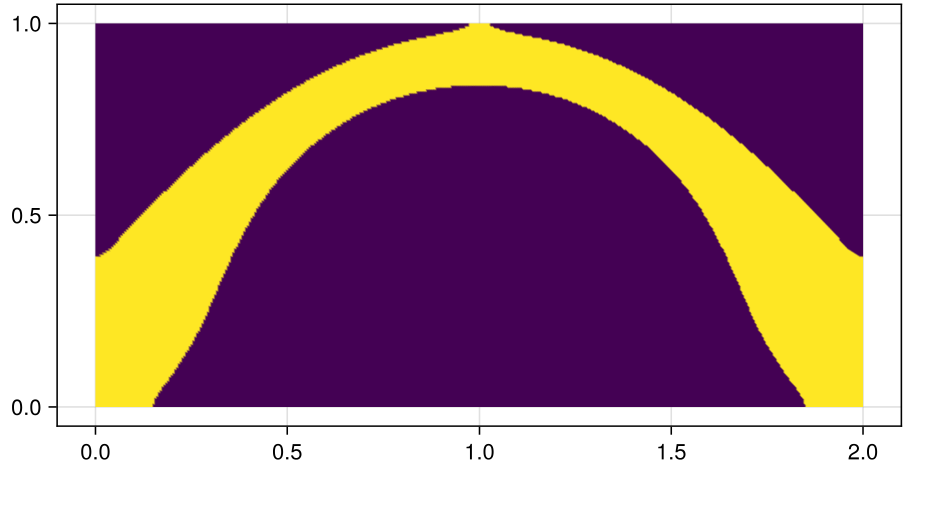}
		\includegraphics[width=0.3\linewidth,trim=1.2cm 18cm 4cm 3cm,clip]{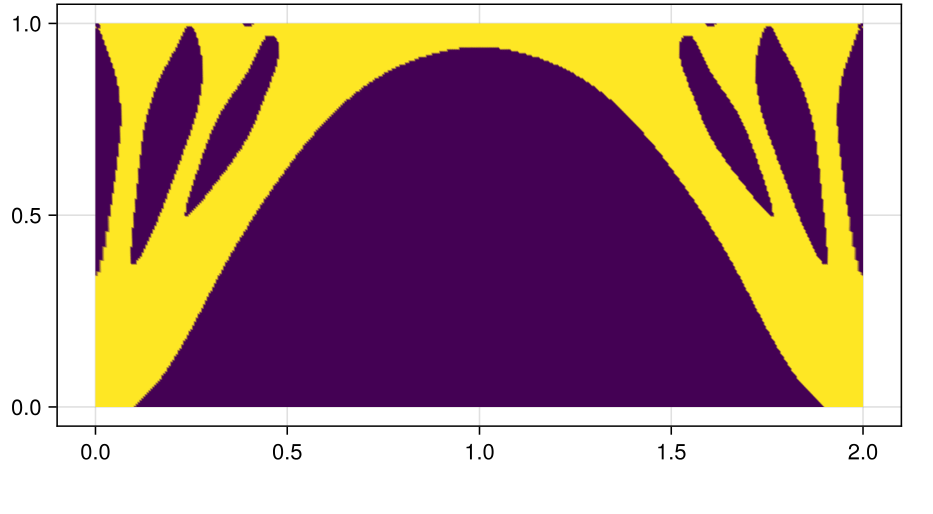}
		\includegraphics[width=0.3\linewidth,trim=1.2cm 18cm 4cm 3cm,clip]{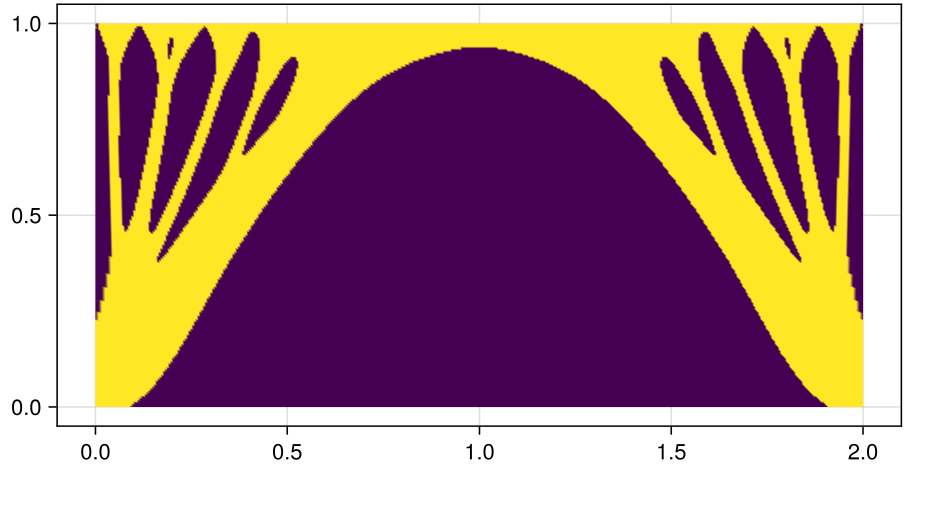}
		\includegraphics[width=0.3\linewidth,trim=2cm 8.5cm 3cm 9.8cm,clip]{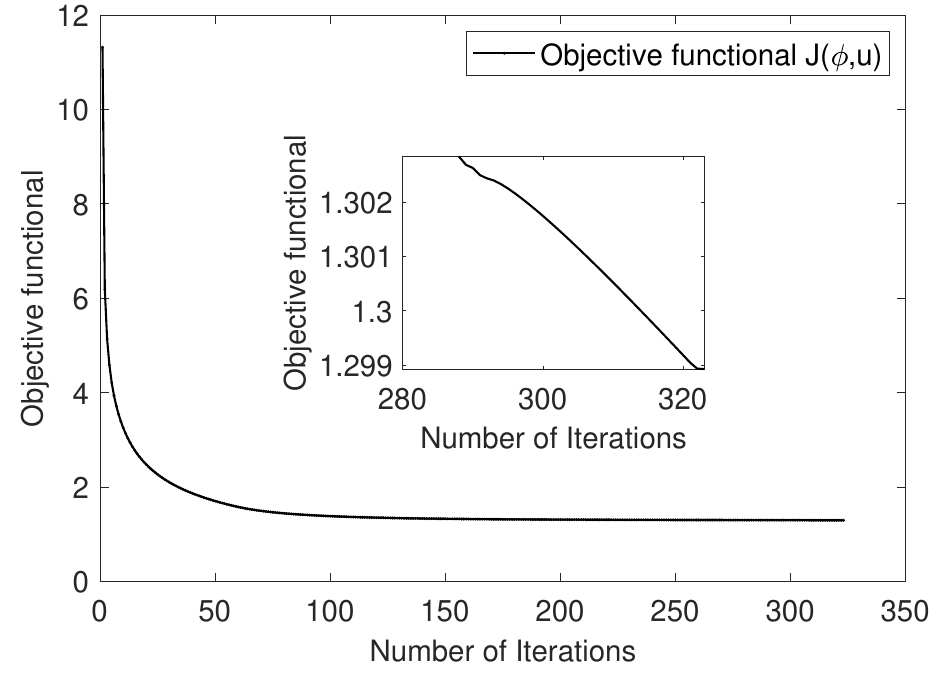}
		\includegraphics[width=0.3\linewidth,trim=2cm 8.5cm 3cm 9.8cm,clip]{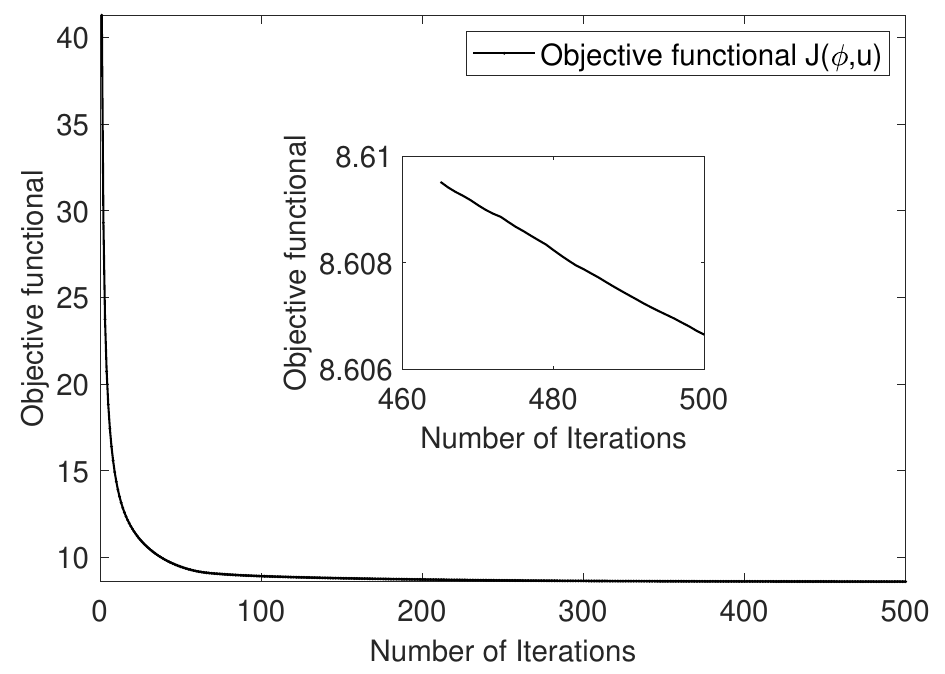}
		\includegraphics[width=0.3\linewidth,trim=2cm 8.5cm 3cm 9.8cm,clip]{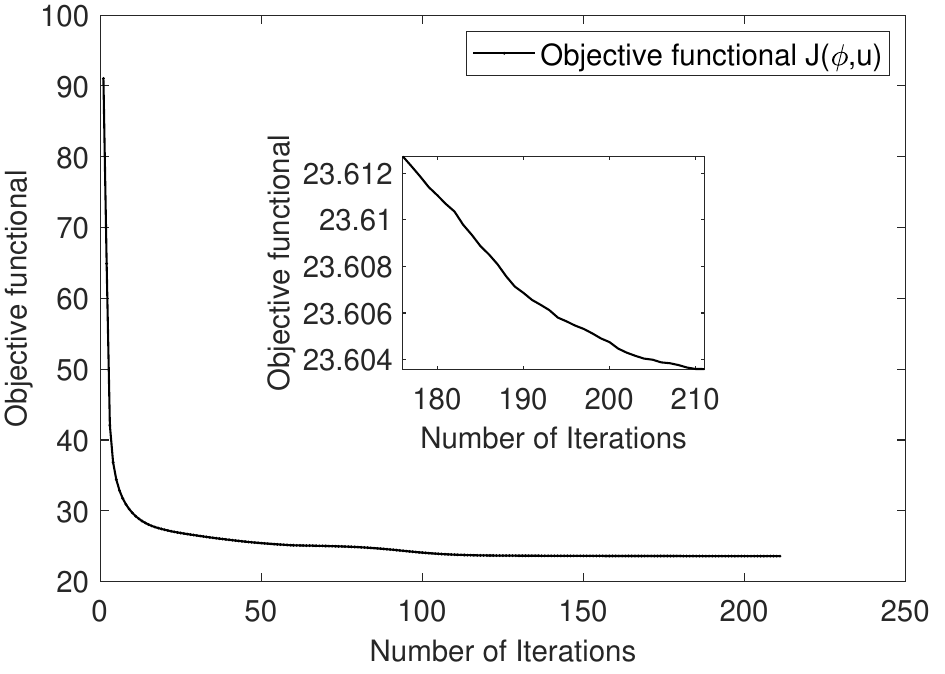}
		\caption{Effects of traction force $\textbf{s}$ on the approximate optimal solutions of $\phi$ and objective functional values on a $400\times 200$ grid with $\gamma = 0.25$, $\beta=0.3$, $\epsilon = 0.002$, $\Delta t = 0.002$, $T=1$. From left to right, $\textbf{s} = (0,0),~(0,-0.5),~(0,-1)$ are used, respectively. See Section~\ref{sec:more}.}
		\label{fig:E5_bodyforce}
	\end{figure}

The algorithm is also applied to the curved boundary region illustrated in Figure \ref{Exa}(e), consisting of two straight edges and two curved edges. Figure \ref{fig:E4_volume} presents the approximate optimal solutions for the phase-field variable $\phi$ under the following parameters: $\epsilon = 0.01$, $\beta = 0.4$, $\gamma = 0.25$ (left) and $0.1$ (right), $\Delta t = 0.01$, and $T = 5$. The initial condition is defined by a uniform random distribution of $\phi$.
	
	\begin{figure}[htbp]
		\centering
		\includegraphics[width=0.3\linewidth,trim=1cm -1.5cm 6cm 0cm,clip]{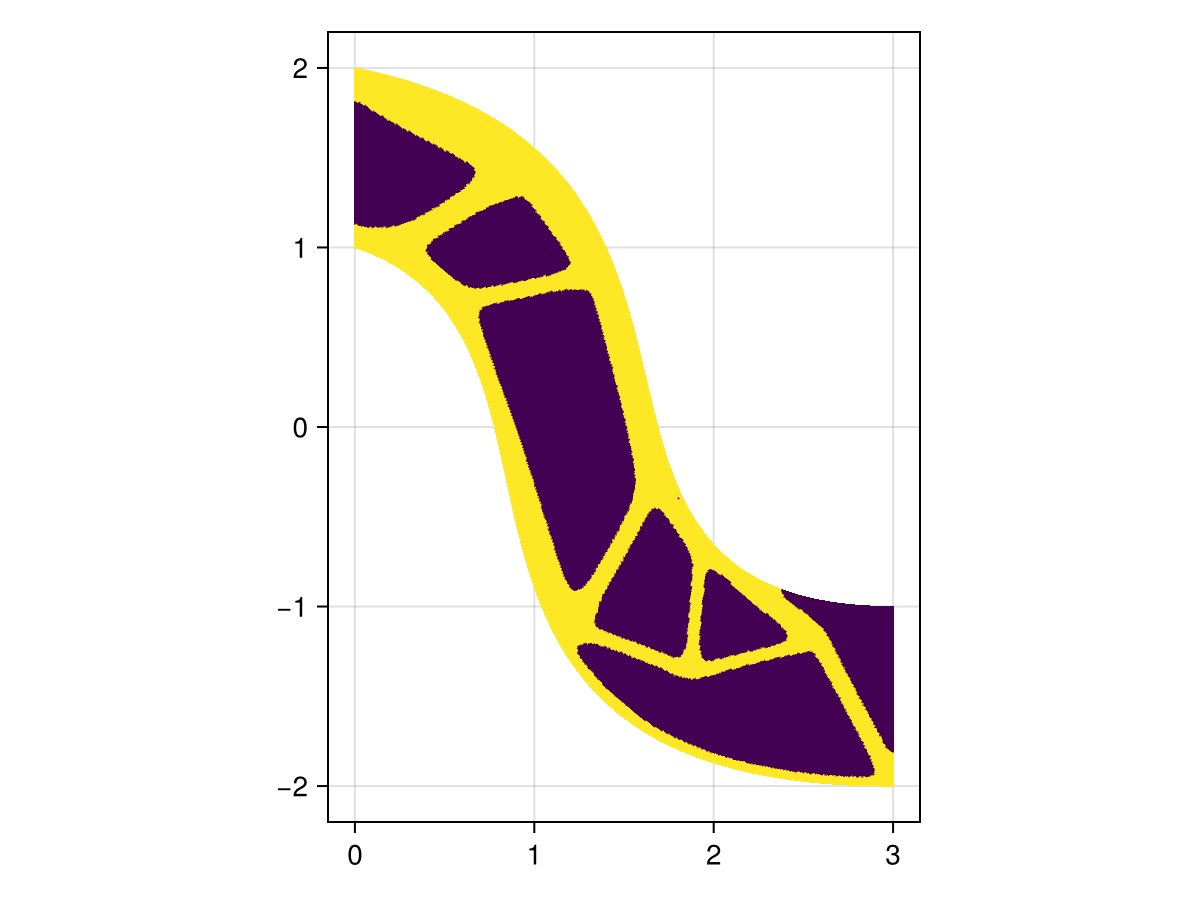}\ \ \ \ \ \ \ \ \ \ 
		\includegraphics[width=0.3\linewidth,trim=1cm -1.5cm 6cm 0cm,clip]{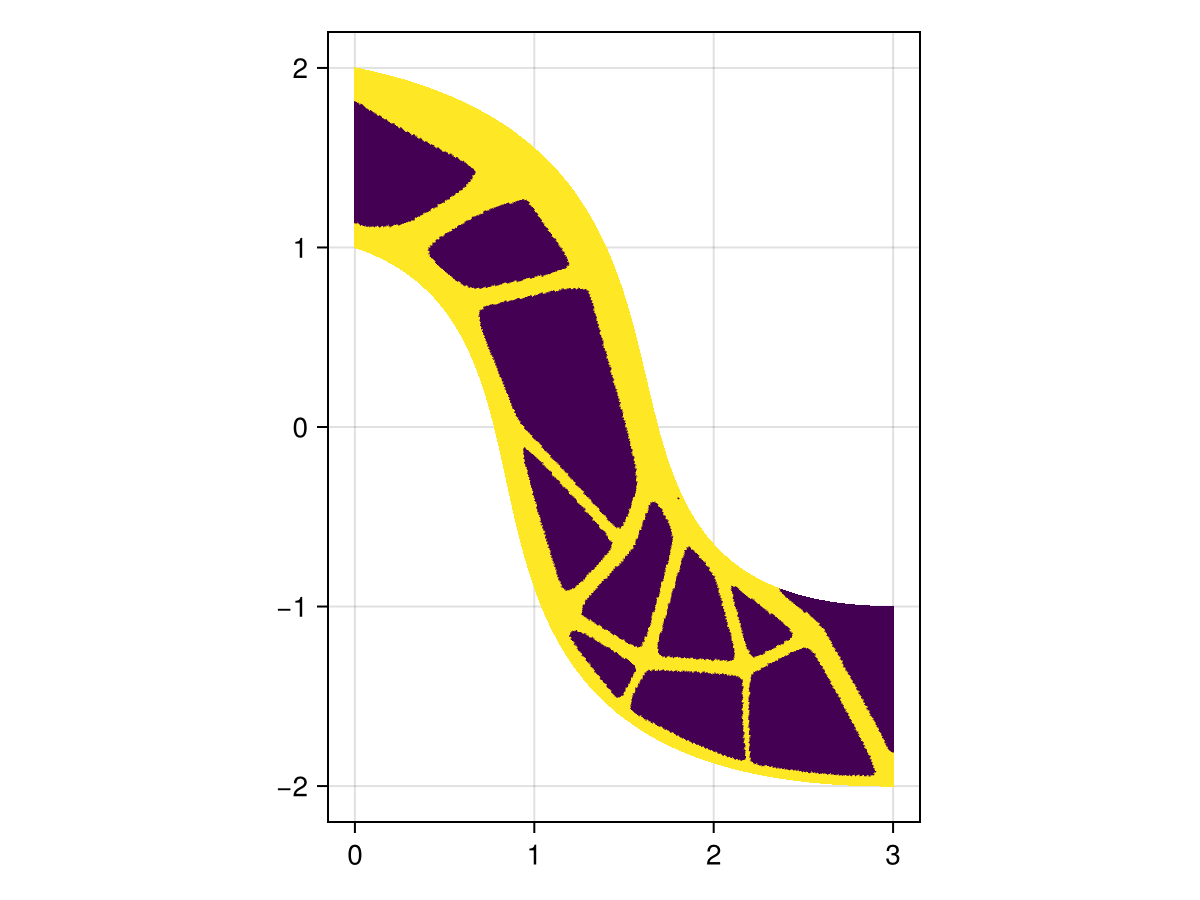}
		\caption{The approximate optimal solutions of $\phi$ with $\epsilon = 0.01$, $\beta = 0.4$ $\Delta t = 0.01$, $T=5$, $\gamma = 0.25$ (left) and $0.1$ (right), and a uniform random initial distribution of $\phi$. See Section~\ref{sec:more}.}
		\label{fig:E4_volume}
	\end{figure}
	
	\subsection{3D examples}\label{E6}
	
	We now present a three-dimensional optimization problem, as illustrated in Figure~\ref{fig:3D}. The model consists of a rectangular cantilever beam clamped on its left side and subjected to a vertical traction force $\mathbf{s} = (0,-1,0)$ applied at the lower right edge. 

The simulations use random initial distributions of  $\phi$ with the following parameters: $\nu = 0.3,~E=1,~\epsilon=0.01,~\Delta t = 0.01, ~T=2$. Figure~\ref{fig:E6} presents three representative results showing the effects of varying the volume fraction $\beta$ and the weighting parameter $\gamma$ on both the optimal phase field solutions and the corresponding objective functional values. The numerical experiments demonstrate behavior consistent with the two-dimensional case: increasing either $\gamma$ or $\beta$ leads to simpler topological configurations, manifested by a reduction in the number of structural branches in the optimal $\phi$ solutions. This dimensional consistency confirms that the observed parameter dependencies are fundamental characteristics of the optimization framework, independent of the spatial dimension being considered.

	\begin{figure}[ht!]
		\centering
		\includegraphics[width=0.3\linewidth]{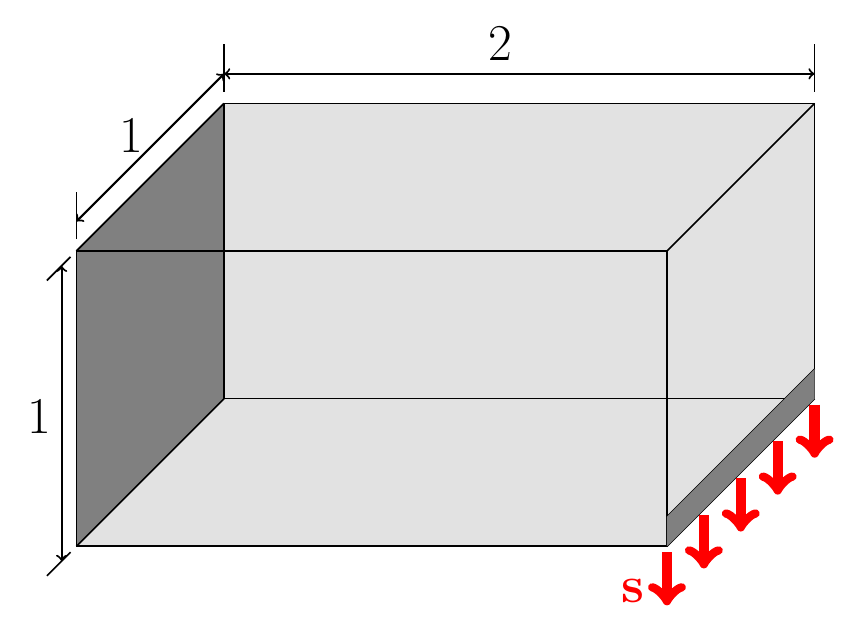}
		\caption{Rectangular cantilever clamped at the left side and loaded at the right by a traction force applied at the lower. See Section \ref{E6}. }
		\label{fig:3D}
	\end{figure}
	
	\begin{figure}[ht!]
		\centering
		\includegraphics[width=0.3\linewidth,trim=1cm 3cm 3cm 0cm,clip]{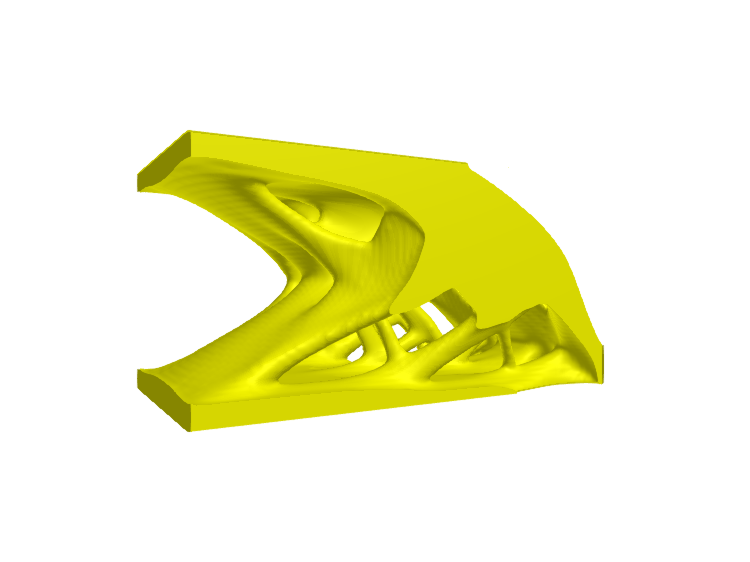}
		\includegraphics[width=0.3\linewidth,trim=1cm 3cm 3cm 0cm,clip]{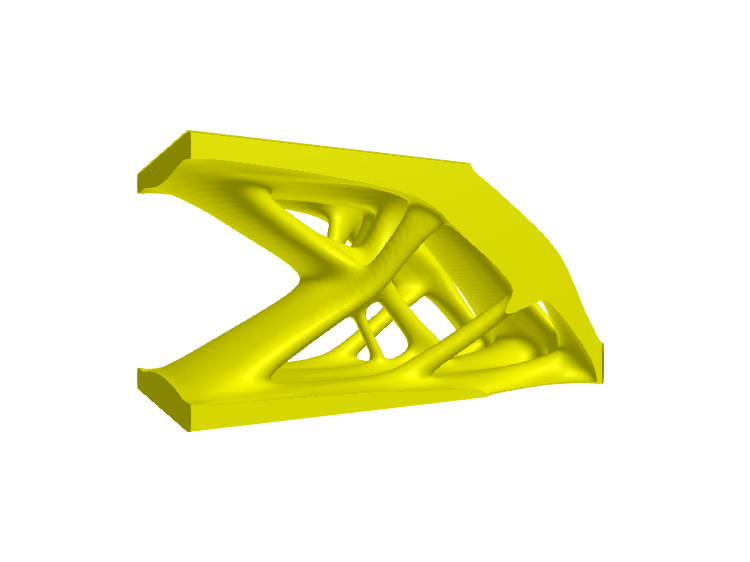}
		\includegraphics[width=0.3\linewidth,trim=1cm 3cm 3cm 0cm,clip]{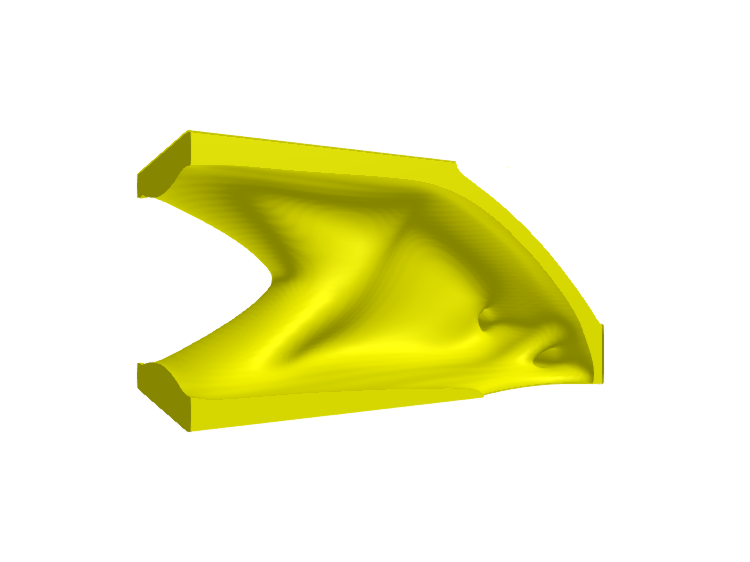}\\
		\includegraphics[width=0.3\linewidth,trim=1cm 8.5cm 3cm 9.8cm,clip]{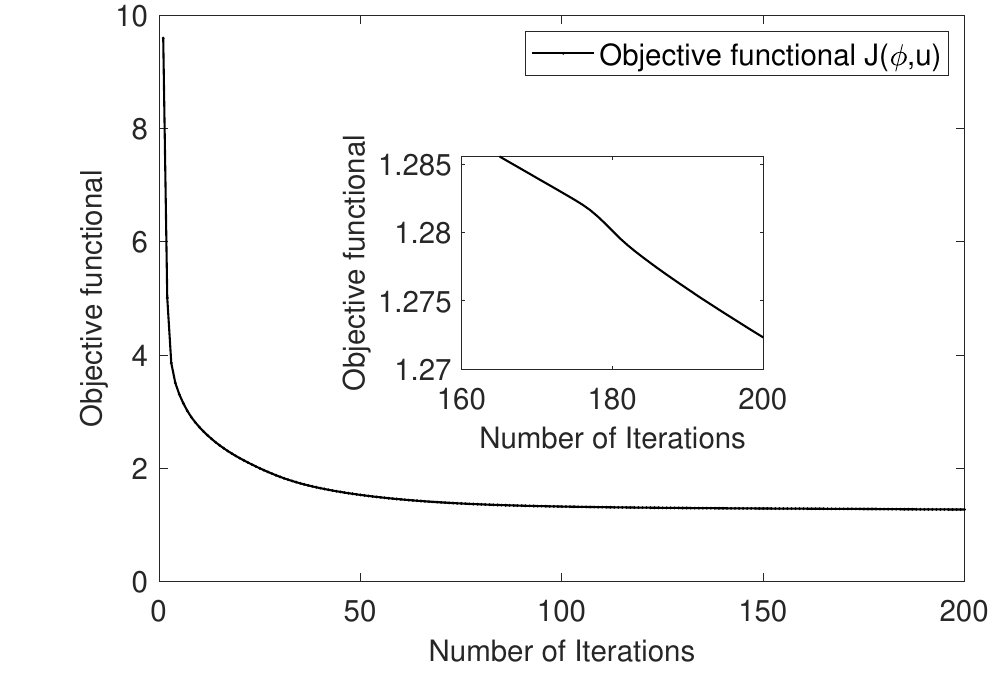}
		\includegraphics[width=0.3\linewidth,trim=1cm 8.5cm 3cm 9.8cm,clip]{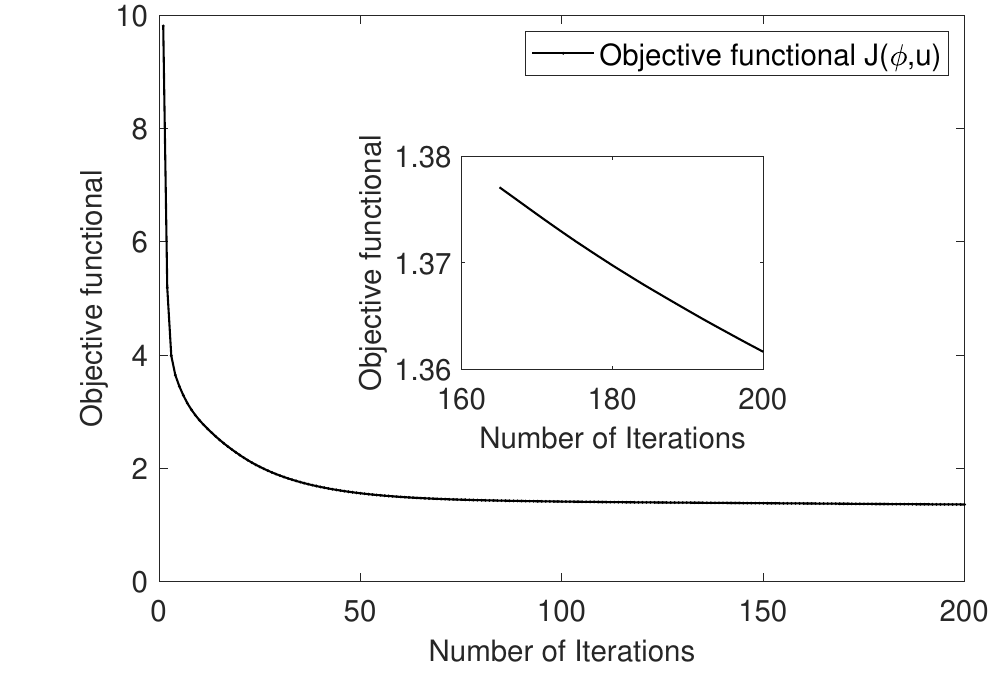}
		\includegraphics[width=0.3\linewidth,trim=1cm 8.5cm 3cm 9.8cm,clip]{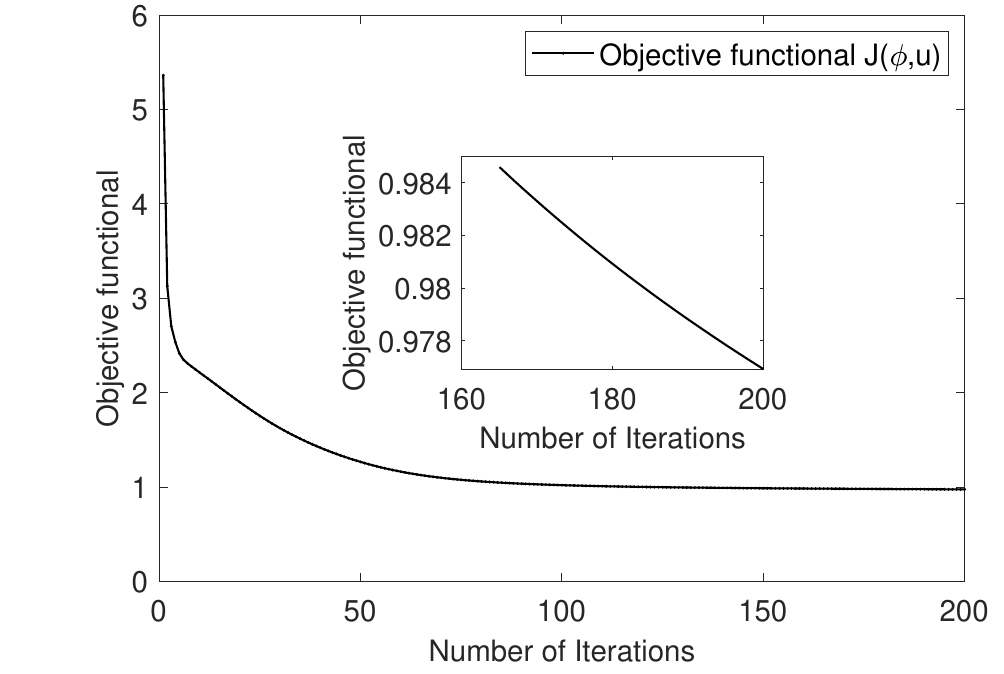}
		\caption{The approximate optimal solutions of $\phi$ and objective functional values. Left, $\gamma=0.1$, $\beta=0.3$; Middle, $\gamma=0.2$, $\beta=0.3$; Right, $\gamma =0.2$, $\beta = 0.4$. See Section \ref{E6}. }
		\label{fig:E6}
	\end{figure}

	\section{Conclusions}\label{sec:include}
	In this work, we have developed a stable phase-field method for topology optimization of minimum compliance problems. The proposed numerical scheme successfully addresses the key challenge of constraint satisfaction while preserving the original optimization objective. Our approach combines three essential components:
\begin{itemize}
\item  A first-order operator splitting method based on Lagrange multipliers for efficient solution of the phase-field equations.
\item A novel limiter mechanism that simultaneously enforces volume constraints and bound-preserving conditions.
\item A stable time discretization that guarantees constraint satisfaction at each iteration.
\end{itemize}

Numerical experiments demonstrate that our method achieves accurate and stable solutions for classical minimum compliance problems. The results confirm the effectiveness of our constraint-preserving approach and its ability to produce physically meaningful optimal designs.

Looking forward, the proposed framework shows significant potential for extension to more complex problems, particularly: multi-material structural topology optimization, fluid-structure interaction problems, nonlinear material response problems. These extensions would further validate the robustness and versatility of our phase-field approach while expanding its range of engineering applications.

	\section*{Acknowledgement}
H. Chen was supported by National Key Research and Development Project of China (Grant No. 2024YFA1012600) and National Natural Science Foundation of China (Grant No. 12471345, 12122115). D. Wang was partially supported by National Natural Science Foundation of China (Grant No. 12422116), Guangdong Basic and Applied Basic Research Foundation (Grant No. 2023A1515012199), Shenzhen Science and Technology Innovation Program (Grant No. RCYX20221008092843046,JCYJ20220530143803007).  X.-P. Wang was partially supported by National Natural Science Foundation of China (Grant No. 12271461, 12426307) and Shenzhen Science and Technology Innovation Program (Grant No. 2024SC0020). D. Wang and X.-P. Wang were also supported by Guangdong Provincial Key Laboratory of Mathematical Foundations for Artificial Intelligence (2023B1212010001), and Hetao Shenzhen-Hong Kong Science and Technology  Innovation Cooperation Zone Project (No.HZQSWS-KCCYB-2024016).
	
	\bibliographystyle{plainnat}
	\bibliography{sample}

\end{document}